\documentclass[11pt,reqno]{amsart}
\usepackage{amssymb}
\usepackage{amsmath}
\usepackage{mathrsfs}
\usepackage{amssymb, amsmath, amsfonts}
\usepackage{bbm}
\usepackage{mathabx,url}

\usepackage{color}

\usepackage[hidelinks]{hyperref}

\makeatletter
\def\tank#1{\protected@xdef\@thanks{\@thanks
 \protect\footnotetext[0]{#1}}}
\def\bigfoot{

 \@footnotetext}
\makeatother

\newcommand{\ea}{\end{array}}

\allowdisplaybreaks
\numberwithin{equation}{section}
\newtheorem{theorem}{Theorem}[section]
\newtheorem{lemma}{Lemma}[section]
\newtheorem{proposition}{Proposition}[section]

\newtheorem{cor}{Corollary}[section]
\newtheorem{remark}{Remark}[section]

\def\beq{\begin{equation}}
\def\nneq{\end{equation}}

\def\bthm{\begin{theorem}}
\def\nthm{\end{theorem}}

\def\blem{\begin{lemma}}
\def\nlem{\end{lemma}}
\def\bprf{\begin{proof}}
\def\nprf{\end{proof}}
\def\bprop{\begin{prop}}
\def\nprop{\end{prop}}
\def\brmk{\begin{rem}}
\def\nrmk{\end{rem}}
  
\def\bexa{\begin{exa}}
\def\nexa{\end{exa}}
\def\bcor{\begin{cor}}
\def\ncor{\end{cor}}

\def\Var{\mathrm {Var}}

\newcommand{\ep}{\varepsilon}

\title[Generalized Fractional Brownian Motion]{Exact uniform 
modulus of continuity  
and  Chung's LIL for the generalized fractional Brownian motion}
  
\author[R. Wang]{Ran Wang}
\address[]{Ran Wang, School of Mathematics and Statistics,  Wuhan University,  Wuhan, 430072,  
China.}
\email{rwang@whu.edu.cn}

\author[Y. Xiao]{Yimin Xiao*} \thanks{*Corresponding author. E-mail: xiaoy@msu.edu.}
\address[ ]{Yimin Xiao, Department of Statistics and Probability, Michigan State University, East Lansing, 
MI 48824, USA.}\email{xiaoy@msu.edu}

\date{}
\begin{document}
\maketitle

 \noindent {\bf Abstract:} 
The generalized fractional Brownian motion (GFBM) $X:=\{X(t)\}_{t\ge0}$ with parameters $\gamma \in [0, 1)$ 
and $\alpha\in \left(-\frac12+\frac{\gamma}{2}, \,   \frac12+\frac{\gamma}{2} \right)$ is a centered Gaussian 
$H$-self-similar process introduced by Pang and Taqqu (2019)
as the scaling limit of power-law shot noise processes, where $H = \alpha-\frac{\gamma}{2}+\frac12 \in(0,1)$. 
When $\gamma = 0$, $X$  is the ordinary fractional Brownian motion. When $\gamma \in (0, 1)$,  GFBM $X$ 
does not have stationary increments, and its sample path properties such as H\"older continuity, path 
differentiability/non-differentiability, and the functional law of the iterated logarithm (LIL) have been investigated 
recently by Ichiba, Pang and Taqqu (2021).  They mainly focused on sample path properties that are described in 
terms of the self-similarity index $H$ (e.g., LILs at infinity or at the origin). 
  
In this paper, we further study the sample path properties of GFBM $X$ and establish the exact uniform modulus 
of continuity,  small ball probabilities, and Chung's laws of iterated logarithm at any fixed point $t > 0$. Our 
results show that the local regularity properties away from the origin and fractal properties  of GFBM $X$ are 
determined by the index $\alpha +\frac1 2$, instead of the self-similarity index $H$. This is in contrast with the 
properties of ordinary fractional Brownian motion whose local and asymptotic properties are determined by the 
single index $H$. 
 \vskip0.3cm

 \vskip0.3cm

\vskip0.3cm

 \section{Introduction and main results}

%%%%%%%%%%%%%%%%%%%%%%%%%%%%%%%%%%%%%%%%%%%%%%%%%%%%%%%%%%%%

The generalized fractional Brownian motion (GFBM, for short) $X:=\{X(t)\}_{t\ge0}$ is a centered Gaussian  self-similar 
process introduced by Pang and Taqqu \cite{PT2019} as the scaling limit of power-law shot noise processes. It has 
the following stochastic integral representation:
 \begin{align}\label{eq X}
 \{X(t)\}_{t\ge0}\overset{d}{=}&\left\{  \int_{\mathbb R}  \left((t-u)_+^{\alpha}-(-u)_+^{\alpha} \right) |u|^{-\gamma/2}
  B(du)  \right\}_{t\ge0},
 \end{align}
 where the parameters $\gamma$ and $\alpha$ satisfy
 \begin{align}\label{eq constant}
 \gamma\in [0,1),  \ \  \alpha\in \left(-\frac12+\frac{\gamma}{2}, \  \frac12+\frac{\gamma}{2} \right),
 \end{align}
and where $B(du)$ is a Gaussian random measure in $\mathbb R$ with the Lebesgue  control measure $du$.  
It follows that the Gaussian process $X$ is self-similar with index $H$ given by  
 \begin{align}
 H=\alpha-\frac{\gamma}{2}+\frac12\in(0,1).
 \end{align}
When $\gamma=0$, $X$ becomes an ordinary fractional Brownian motion  (FBM) $B^H$ which 
can be represented as:
\begin{align}\label{eq FBM}
\left\{B^H(t)\right\}_{t\ge 0}\overset{d}{=}\left\{ \int_{\mathbb R}  
\left((t-u)_+^{H-\frac12}-(-u)_+^{H-\frac12} \right) B(du)  \right\}_{t\ge0}.
\end{align} 
However, when $\gamma \ne 0$, $X$ does not have the property of stationary increments.   
 
Fractional Brownian motion $B^H$ has been studied extensively in the literature. It is well known that $B^H$ arises 
naturally as the scaling limit of many interesting stochastic systems. For example,  \cite{KK2004} or  
\cite[Chapter 3.4]{PipTaq2017} showed that the scaled power-law Poisson shot noise process with stationary 
increments converges to $B^H$. Pang and Taqqu \cite{PT2019} studied a class of integrated shot-noise processes 
with power-law non-stationary conditional variance functions and proved in their Theorem 3.1 that the corresponding 
scaled process converges weakly to GFBM $X$.  

As shown by Pang and Taqqu \cite{PT2019},  GFBM $X$ is a natural generalization of the ordinary FBM. 
It preserves the self-similarity property while the factor $|u|^{-\gamma/2}$ introduces non-stationarity of increments, 
which is useful for reflecting the non-stationarity in physical stochastic systems. Ichiba, Pang and Taqqu \cite{IPT2020} 
raised the interesting question: ``How does the  parameter $\gamma$ affect the sample path properties of  GFBM?" 
They proved in \cite{IPT2020}  that, for any $T>0$ and $\ep>0$, the sample paths of  $X$  are H\"older continuous  in $[0,T]$ of 
order $H-\ep$ and the functional and local laws of the iterated logarithm of $X$ are determined by the self-similarity 
index $H$.  More recently,  Ichiba, Pang and Taqqu \cite{IPT2020b} studied the semimartingale properties of   
GFBM $X$ and its mixtures and applied them to model the volatility processes in finance.

In this paper, we study precise local sample path properties of  GFBM $X$, including the exact uniform modulus of 
continuity,  small ball probabilities,   Chung's law of iterated logarithm at any fixed point $t > 0$ and the tangent 
processes. Our main results 
are Theorems \ref{thm uniform}-\ref{thm tangent} below. They show that the local regularity properties of GFBM $X$ 
away from the origin are determined by the index $\alpha +\frac1 2$, instead of the self-similarity index $H= \alpha
-\frac{\gamma}{2}+\frac12$. Our results also imply that the fractal properties of GFBM $X$ are determined by 
$\alpha +\frac1 2$, see Remarks \ref{Re:1} (ii) and \ref{Re:refine} below. This is in contrast with the ordinary fractional 
Brownian motion whose local, fractal, and asymptotic properties are determined by the single index $H$. 
We remark that our results are also useful for studying other fine sample path properties of GFBM $X$. For example,  
one can determine the exact Hausdorff measure functions for various fractals generated by the sample paths, and 
prove sharp H\"older conditions and tail probability estimates for the local times of GFBM $X$ as in, e.g.,  
\cite{Tal95, Tal98, Xiao1997,Xiao2007, Xiao2009}.

The first result is related to Theorems 3.1 and 4.1 of Ichiba, Pang and Taqqu \cite{IPT2020} and 
provides the exact uniform moduli of continuity for $X$ and its derivative $X'$ (when it exists)
in $[a, b]$, where $0 < a < b< \infty$ are constants.

 \begin{theorem}\label{thm uniform}
 Let $X:=\{X(t)\}_{t\ge0}$ be the GFBM defined in (\ref{eq X}) and let $0 < a < b< \infty$ be constants.
 \begin{itemize}
\item[(a).] \, If $\alpha\in(-1/2+\gamma/2, 1/2)$, then  there exists a constant $\kappa_1\in(0,\,\infty)$ 
such that 
 \begin{align}\label{eq unif X}
 \lim_{r\rightarrow0+} \sup_{a\le t\le b}\sup_{0\le h\le r}\frac{|X(t+h)-X(t)|}{h^{\alpha+\frac12}\sqrt{\ln h^{-1}}} 
 =\kappa_1, \ \  \ \text{a.s.}
 \end{align} 
 \item[(b).] \, If $\alpha= 1/2$, then  there exists a constant $\kappa_2\in(0,\, \infty)$ 
such that 
 \begin{align}\label{eq unif X2}
 \lim_{r\rightarrow0+} \sup_{a\le t\le b}\sup_{0\le h\le r}\frac{|X(t+h)-X(t)|}{h  \ln h^{-1}} 
 \le\kappa_2, \ \  \ \text{a.s.}
 \end{align}  
 \item[(c).]\, If $\alpha\in (1/2, 1/2+\gamma)$, then $X$ has a modification that 
 is continuously differentiable on $[a,b]$ and its derivative $X'$ satisfies that
 \begin{align}\label{eq unif X'}
 \lim_{r\rightarrow0+} \sup_{a\le t\le b}\sup_{0\le h\le r}\frac{|X'(t+h)-X'(t)|}{h^{\alpha-\frac12}\sqrt{\ln h^{-1}}} 
 =\kappa_3, \ \  \ \text{a.s.},
 \end{align} 
 where  $\kappa_3\in(0,\,\infty)$ is a constant.
 \end{itemize}
 \end{theorem}

 \begin{remark} {\rm
 \begin{itemize}
 \item[(i).]  Theorem 3.1 in \cite{IPT2020} states that, for all $\varepsilon>0$,  $X$ has a modification 
that satisfies the uniform H\"older condition in $[0, T]$ of order  $ \alpha-\gamma/2+1/2-\varepsilon$.  
Our Theorem \ref{thm uniform} shows that the sample path of $X$ on any interval $[a,\, b]$ with $a> 0$
is smoother than its behavior at $t = 0$, which is determined by the  self-similarity index $ H= 
\alpha-\gamma/2+1/2$ as suggested by 
$
\mathbb E\left[X(t)^2\right]=c(\alpha, \gamma) t^{2H},
$
with
 \begin{align}\label{eq c}
 c(\alpha, \gamma)= \mathcal B(2\alpha+1, 1-\gamma)+ \int_0^{\infty}((1+u)^{\alpha}-u^{\alpha})^2 u^{-\gamma}du.
 \end{align} 
 Here and below, $\mathcal B(\cdot,\cdot)$ denotes the Beta function.
\item[(ii).] We believe that the equality in (\ref{eq unif X2}) holds. However, we have not been able to prove this. The 
reason is that, when $\alpha = 1/2$, the lower bounds in Lemma \ref{Lem:VarioZ} and Proposition \ref{lem SLND} 
are different. Similarly, the case of $\alpha = 1/2$ is excluded in Theorems \ref{thm Xsb}-\ref{thm LIL0} below.
 \end{itemize}
}
\end{remark}

The next two results are on the small ball probabilities of $X$. They show a clear difference for the 
two cases of $s \in [0, r]$ 
and $s\in [t-r, t + r]$ with $t >r> 0$. The small ball probabilities are not only useful for proving Chung's law 
of iterated logarithm (Chung's LIL, in short) in Theorem \ref{thm LIL0} but also have many other applications.
We refer to Li and Shao \cite{LiShao01} for more information.
  
 \begin{theorem}\label{thm Xsb}
Assume $\alpha\in (-1/2+\gamma/2, \,1/2)$. Then there exist    constants $\kappa_4, \kappa_5\in (0,\infty)$ such that for all $r>0$ and $ 0<\varepsilon<1$, 
 \begin{align}\label{eq small X}
\exp\bigg(- \kappa_{4}  \Big(\frac{r^H}{\varepsilon}\Big)^{\frac{1}{\alpha+1/2}} \bigg)\le 
 \mathbb P\bigg\{\sup_{s\in [0,r]} |X(s)|\le  \varepsilon \bigg\} 
 \le  \exp\bigg(- \kappa_5 \Big(\frac{r^H}{\varepsilon}\Big)^{\frac{1}{\alpha+1/2}} \bigg),
 \end{align} 
 where $H=\alpha-\gamma/2+1/2$. 
 \end{theorem} 
   
 \begin{theorem}\label{Prop:sbX}
 \begin{itemize}
 \item[(a).]\, Assume $\alpha\in(-1/2+\gamma/2,\,1/2)$.  Then
   there exist   constants $\kappa_6, \kappa_7\in (0,\infty)$  
such that for all $t>0, r \in (0, t/2)$ and $\varepsilon\in \left(0, r^{\alpha+1/2}\right)$,
   \begin{equation} \label{Eq:sb1210}
   \begin{split}
  \exp\bigg(-\kappa_6\,r \, c_1(t)\Big(\frac{1}{\varepsilon}\Big)^{\frac1{\alpha+1/2}}\bigg)\le  &\,
  \mathbb P\bigg\{\sup_{|h|\le r}|X(t+h)-X(t)|\le \varepsilon \bigg\}\\
 &\, \le   \exp\bigg(-\kappa_7\,r\, c_2(t)\Big(\frac{1}{\varepsilon}\Big)^{\frac{1}{\alpha+1/2}}\bigg),
  \end{split}
   \end{equation} 
   where $c_1(t)=\max\left\{t^{\alpha-\gamma/2-1/2}, t^{-\gamma/(2\alpha+1)}\right\}$ and $c_2(t)
   =t^{-\gamma/(2\alpha+1)}$.
   \item[(b).]\,
   Assume    $\alpha\in(1/2,\, 1/2+\gamma/2)$.  Then there 
   exist   constants $\kappa_8, \kappa_9\in (0,\infty)$  
   such that for all $t>0, r \in (0, t/2)$ and $\varepsilon\in \left(0, r^{\alpha-1/2}\right)$,
   \begin{equation} \label{Eq:sb2210}
   \begin{split}
  \exp\bigg(-\kappa_8\, r\, c_3(t)\Big(\frac{1}{\varepsilon}\Big)^{\frac1{\alpha-1/2}}\bigg)\le &
\,   \mathbb P\bigg\{\sup_{|h|\le r}|X'(t+h)-X'(t)|\le \varepsilon \bigg\} \\
  &\, \le   \exp\bigg(-\kappa_9\, r\, c_4(t)\Big(\frac{1}{\varepsilon}\Big)^{\frac{1}{\alpha-1/2}}\bigg),
   \end{split}
   \end{equation}  
    where $c_3(t)=\max\left\{t^{\alpha-\gamma/2-3/2}, t^{-\gamma/(2\alpha-1)}\right\}$ and $c_4(t)
    =t^{-\gamma/(2\alpha-1)}$.
   \end{itemize}
 \end{theorem}

The following are Chung's laws of the iterated logarithm for $X$ and $X'$. It is interesting to notice that
the parameters $\gamma$ and $\alpha$ play different roles. Since $X$ and $X'$ do not have stationary 
increments when $\gamma > 0$, the limits in their Chung's LILs depend on the location of $t>0$. 
\eqref{eq LIL01} and \eqref{eq LIL01b} show that the oscillations decrease at the rate 
$t^{-\gamma/2}$ as $t$ increases. This provides an explicit answer in the context of Chung's LIL to 
the aforementioned question of Ichiba, Pang and Taqqu \cite{IPT2020} regarding the effect of the parameter
$\gamma$.

\begin{theorem}\label{thm LIL0}
\begin{itemize}
       \item[(a).]\,
If $\alpha\in(-1/2+\gamma/2,1/2)$, then there exists a constant $\kappa_{10}\in(0,\infty)$ 
 such that for every $t>0$, 
    \begin{align}\label{eq LIL01}
   \liminf_{r\rightarrow0+} \sup_{ |h|\le r}\frac{ |X(t+h)-X(t)|} {r^{\alpha+1/2}/(\ln \ln 1/r)^{\alpha+1/2}}
   =\kappa_{10} t^{-\gamma/2},   \ \ \ \text{a.s.}
   \end{align} 
  \item[(b).]\, 
 If $\alpha\in(1/2, \, 1/2+\gamma/2)$,   then there exists a constant $\kappa_{11}\in(0,\infty)$  
 such that  for every $t>0$,   
    \begin{align}\label{eq LIL01b}
   \liminf_{r\rightarrow0+} \sup_{ |h|\le r}\frac{ |X'(t+h)-X'(t)|} {r^{\alpha-1/2}/(\ln \ln 1/r)^{\alpha-1/2}}
   =\kappa_{11} t^{-\gamma/2}, \  \ \ \ \text{a.s.}
   \end{align}     
\end{itemize}
 \end{theorem}  

\begin{remark} {\rm From the proofs of Theorems   \ref{Prop:sbX} and  \ref{thm LIL0}, we know that if $\sup_{ |h|\le r}$ is 
replaced by $\sup_{ 0\le h\le r}$ in \eqref{Eq:sb1210}-\eqref{eq LIL01b}, then the 
corresponding results also hold.}
\end{remark}
   
For completeness, we also include the following law of the iterated logarithm for GFBM $X$ at any fixed point $t > 0$. 
Part (a) of our Theorem  \ref{thm LIL2} supplements Theorem 6.1 in \cite{IPT2020} where the case of $t = 0$ was considered. 
See also Proposition \ref{prop LIL0}  at the end of the present paper for a slight improvement of \cite[Theorem 6.1]{IPT2020}  
using the time inversion property of GFBM. Theorems \ref{thm LIL0} and \ref{thm LIL2} together describe precisely the 
large and small oscillations of $X$ in the neighborhood of every fixed point $t > 0$.  As such they are useful for studying 
fine fractal properties (such as exact Hausdorff measure function and multifractal structure) of the sample path of $X$. 

The topic of LILs for Gaussian processes has been studied extensively by many authors, see for example, Arcones \cite{Arc},
Marcus and Rosen \cite{MR}, Meerschaert, Wang and Xiao \cite{MWX}. In particular, Chapter 7 of \cite{MR} provides explicit 
information about the constant in LIL for Gaussian processes with stationary increments under extra regularity conditions 
on the variance of the increments or the spectral density functions. However, the regularity conditions in \cite{MR} are not 
easy to verify for the stationary processes obtained from GFBM $X$, so the results in \cite{MR} can not be applied directly. 
Instead, we will make use of Theorem 5.1 in \cite{MWX} and the stationary Gaussian process $U$ in Section \ref{sec:Lamperti} 
to prove the following theorem.  As in Theorem \ref{thm LIL0},  our result below describes explicitly the roles played by the 
 parameters $\gamma$ and $\alpha$ and the location $t>0$.

\begin{theorem}\label{thm LIL2}
\begin{itemize}
\item[(a).]\,
If $\alpha\in(-1/2+\gamma/2,1/2)$, then there exists a constant $\kappa_{12}\in(0,\infty)$ such that  for every $t>0$,   
   \begin{align}\label{eq LIL21}
   \limsup_{r\rightarrow0+} \sup_{ |h|\le r}\frac{ |X(t+h)-X(t)|} {r^{\alpha+1/2} \sqrt{\ln \ln 1/r}}=\kappa_{12} t^{-\gamma/2}, 
   \  \ \ \ \text{a.s.}
   \end{align} 
   \item[(b).]\,  If $\alpha =1/2$, then there exists a constant $\kappa_{13}\in(0,\infty)$ such that  for every $t>0$,   
   \begin{align}\label{eq LIL21a}
   \limsup_{r\rightarrow0+} \sup_{ |h|\le r}\frac{ |X(t+h)-X(t)|} {r \sqrt{\ln (1/r) \ln \ln 1/r}} = \kappa_{13}t^{-\gamma/2},  
   \ \ \ \text{a.s.}
   \end{align} 
  \item[(c).]\,If $\alpha\in(1/2, \, 1/2+\gamma/2)$, then  there exists a constant $\kappa_{14}\in(0,\infty)$    
 such that  for every $t>0$,
    \begin{align}\label{eq LIL21b}
   \limsup_{r\rightarrow0+} \sup_{ |h|\le r}\frac{ |X'(t+h)-X'(t)|} {r^{\alpha-1/2}\sqrt{\ln \ln 1/r}}=\kappa_{14} t^{-\gamma/2}, 
   \  \ \ \ \text{a.s.}
   \end{align}     
\end{itemize}
\end{theorem}    

In order to prove the theorems stated above, we consider the following decomposition of $X(t)$ for all $t \ge 0$:
 \begin{equation}\label{eq decom}
 \begin{split}
  X(t) &=\int_{-\infty}^0  \left((t-u)^{\alpha}-(-u)^{\alpha} \right) (-u)^{-\gamma/2} B(du)  
  +\int_0^t  (t-u)^{\alpha}   u^{-\gamma/2} B(du)\\
  &=: Y(t) + Z(t).
\end{split}
\end{equation}
Then the two processes  $Y=\{Y(t)\}_{t\ge0}$ and  $Z=\{Z(t)\}_{ t\ge0}$ are independent.  The process 
$Z$ in \eqref{eq decom} is well defined for $\alpha > -1/2$ and $\gamma \in (0, 1)$ and is called a 
{\it generalized Riemann-Liouville FBM}, following the terminology of Ichiba, Pang and Taqqu \cite{IPT2020}. 
Notice that the range of the parameter $\alpha$ for $Z$ is wider than that in  \eqref{eq constant}.  As in 
\cite[Remark 5.1]{PT2019}, one can verify that $Z$ is a self-similar Gaussian process with Hurst index 
$H=\alpha-\frac{\gamma}{2}+\frac12$ which is negative if $\alpha \in (-1/2,\, -1/2+\gamma/2)$. It follows 
from Lemma  \ref{Lem:VarioZ}   below that $Z$ has a modification whose sample 
function is continuous on $(0, \infty)$ a.s. In Section 2, we will prove that $Y$  has a modification that is 
continuously differentiable of all orders in $(0,\infty)$. Hence, in order to study the regularity properties 
of $X$, we only need to study in detail the regularity properties of  the sample path of $Z$.  

Intuitively, if $u\in [a,b]\subset (0,\infty)$,  the perturbation of $u^{-\gamma/2}$ is bounded and  it  does not 
deeply affect the sample path properties of $Z(t)$. Consequently, the process $Z$ shares many regularity 
properties of the following process:
\begin{align}\label{eq RLFBM}
  Z^{\alpha, 0}(t):=\int_0^t  (t-u)^{\alpha} B(du),\  \ \ \alpha>-1/2,
\end{align}
which is the {\it Riemann-Liouville FBM} introduced by L\'evy \cite{L53}, see also Mandelbrot and Van Ness 
\cite{MN1968}, Marinucci and Robinson \cite{MR99} for further information. When $\alpha
\ge 1$ is a positive integer, then $Z^{\alpha, 0}$ is, up to a constant factor, an $\alpha$-fold primitive of 
Brownian motion and its precise local asymptotic properties were studied  by Lachal \cite{Lachal97}.
 
Theorems \ref{thm LIL0}-\ref{thm LIL2} demonstrate that, since GFBM $X$ has non-stationary increments, 
the local oscillation properties of $X$ near a point $t\in (0,\infty)$ are location dependent. As suggested by 
one of the referees,\footnote{We thank the referee for raising this interesting question.}  another way to study 
the local structure of $X$ near $t$ is to determine the limit (in the sense of all the finite dimensional distributions 
or in the sense of weak convergence) of the following sequence of scaled enlargements of $X$ around $t$:  
 \begin{align}\label{eq ssl}
 \bigg\{\frac{X(t+r_n\tau)-X(t)}{c_n} \bigg\}_{\tau \ge 0},
 \end{align}
where $ \{r_n\}$ and $ \{c_n\}$ are sequences of real numbers such that $r_n\searrow0$ and
$c_n\searrow0$. The (small scale) limiting process of \eqref{eq ssl}, when it exists, is called a tangent process  
of $X$ at $t$ by Falconer \cite{Fal02, Fal03}. If the limit in \eqref{eq ssl} exists for $c_n = r_n^\chi$ for some 
constant $\chi \in (0, 1]$ which may depend on $t$, one also says that $X$ is weakly locally asymptotically self-similar 
of order $\chi$ at $t$ (cf.  \cite{BJR97}). Tangent processes of stochastic processes and random fields are 
useful in both theoretical and applications. We refer to \cite{Fal02, Fal03} for a general framework  of tangent 
processes and to \cite{BS11,BS13} for some statistical applications. Recently, Skorniakov \cite{Sko19} provided 
sufficient and necessary conditions for a class of self-similar Gaussian processes to admit a unique tangent process
at any fixed point $t>0$. The theorems in \cite{Sko19} are applicable to the Riemann-Liouville 
FBM,  sub-fractional Brownian motion, and bi-fractional Brownian motion. 
However,  it is not obvious to verify that GFBM $X$ satisfies the conditions of Theorem 2.1 of 
Skorniakov \cite{Sko19}.  
In the following we provide some results on the tangent processes of GFBM $X$ by using direct arguments.

For any fixed $t>0,\, u\ge0$ and a constant $\chi \in (0, 1]$, define the scaled process $\big\{V^{t,u}(\tau)\big\}_{\tau\ge0}$ 
around $t$ by 
 \begin{equation}\label{eq V}
 V^{t,u}(\tau):=\frac{ X(t+u\tau)-X(t)}{u^\chi}, \ \ \ \tau\ge0.
 \end{equation} 
 \begin{itemize}
 \item[(a).]
 If $t = 0$, then, by the self-similarity of $X$, we take $\chi = H$ and see that the corresponding tangent process 
 is $X$ itself.
 \item[(b).] If $t > 0$, then the choice of $\chi$ depends on the parameter $\alpha$.
 \begin{itemize}
 \item[(b1).]
  If $\alpha\in (1/2, 1/2+\gamma/2)$, 
 then by the differentiability of $X$ (see also \cite[Example 2]{Fal02}), we can take $\chi = 1$ and derive that, 
 for any $t > 0$, 
 $$\big\{V^{t,u}(\tau)\big\}_{\tau\ge0} \text{ converges in distribution to }\ \big\{c_{1,1} t^{H-1} B^1(\tau)\big\}_{\tau\ge0}$$ 
 in $C(\mathbb R_+, \mathbb R)$ (the  space of continuous functions from $\mathbb R_+$ to $\mathbb R$), 
 as $u\rightarrow 0+$, where  $B^1(\tau)=\tau\mathcal N$ with $\mathcal N$ 
 being  a standard  Gaussian random variable and $$c_{1,1}=\alpha\left(\int_{-\infty}^{t}(t-u)^{2\alpha-2}|u|^{-\gamma}du\right)^{1/2}.$$  
 \item[(b2).]
 If $\alpha\in (-1/2+\gamma/2,1/2)$, then it follows  from the decomposition \eqref{eq decom} and the 
differentiability of $Y$  (see Proposition \ref{prop Y} below) that the weak limit of $\big\{V^{t,u}(\tau)\big\}_{\tau\ge0}$ as 
$u\rightarrow 0+$ is the same as the analogous scaled process for the generalized Riemann-Liouville FBM $Z$ (See Proposition \ref{Prop tangent Z} below).  
More precisely, we obtain the   result in Theorem \ref{thm tangent} below. 
\end{itemize}
\end{itemize}

 \begin{theorem}\label{thm tangent} Assume $\alpha\in (-1/2+\gamma/2, 1/2), t>0$ and  $\chi = \alpha + \frac 1 2$ in (\ref{eq V}).
   Then the process $\big\{V^{t,u}(\tau)\big\}_{\tau\ge0}$  defined by (\ref{eq V}) converges in distribution to
    $\big\{  \kappa_{15} t^{-\gamma/2} B^{\alpha+1/2}(\tau)\big\}_{\tau\ge 0}$ in $C(\mathbb R_+,\mathbb R)$,   as $u\rightarrow0+$.
 Here $B^{\alpha+1/2}$ is a FBM with index $\alpha+1/2$, and  $$\kappa_{15}=\left( \frac{1}{2\alpha+1}
 +\int_0^{\infty} \big[(1+v)^{\alpha}-v^{\alpha}\big]^2dv\right)^{1/2}.$$
 \end{theorem}

The rest of the paper is organized as follows. In Section 2 we  prove that sample paths of  the process $Y$ 
are smooth in $(0, \infty)$ almost surely.  From Section 3 to Section 6, we focus on the generalized 
Riemann-Liouville FBM $Z$.  More precisely,  we give estimates on the moment of increments,    
prove the existence of the tangent process  and establish 
the one-sided strong local nondeterminism  of $Z$ in Section 3; study the Lamperti transformation of $Z$ and 
give some  spectral density estimates in Section 4;  determine the small ball probabilities for $Z$ in Section 5; 
and prove a Chung's LIL for $Z$ in Section 6.  In Section 7, we prove the main theorems for GFBM $X$.

 \section{Sample path properties of  $Y$}   
 
In this section, we consider the process $Y= \{Y(t)\}_{ t \ge 0}$ defined in   (\ref{eq decom}), namely,  
  \begin{align}\label{eq Y}Y(t) = \int_{-\infty}^0  \left((t-u)^{\alpha}-(-u)^{\alpha} \right) (-u)^{-\gamma/2} B(du),
  \end{align} 
and show that its sample function is smooth away from the origin.   
 
   \begin{lemma}\label{lem Y moment}
   Assume $-1/2+\gamma/2<\alpha<1/2+\gamma/2$. Then    for all $0 < s < t<\infty$, 
   \begin{align}\label{Eq: Ymoment1}
c_{2,1}\frac{|t-s|^2}{t^{2-2H}}\le     \mathbb E\left[\left(Y(t)-Y(s) \right)^2\right]\le c_{2,1}\frac{|t-s|^2}{s^{2-2H}},
 \end{align}  
 where      $c_{2,1}= \alpha^2\int_0^{\infty}(1+u)^{2\alpha-2}u^{-\gamma}du$.
 
 \end{lemma}
    \begin{proof} For   any  $0 < s < t<\infty$,   by Taylor's formula,
     there exists a function $\{\theta(u)\}_{ u\in (0,\infty)}$ 
   valued in $(0,1)$ such that for any $u>0$,
 \begin{equation}\label{eq Taylor1}
 \begin{split}
(t+u)^{\alpha}-(s+u)^{\alpha}=&\alpha\left[s+u+\theta(u)(t-s)\right]^{\alpha-1} (t-s).
 \end{split}
 \end{equation}
           Thus, we have 
     \begin{equation*} 
     \begin{split}
   \mathbb E\left[\left(Y(t)-Y(s) \right)^2\right]=&\int_0^{\infty} \left[(t+u)^{\alpha}-(s+u)^{\alpha}  \right]^2u^{-\gamma}du \\
       \ge &\,\alpha^2(t-s)^2 \int_0^{\infty}(t+u)^{2(\alpha-1)} u^{-\gamma}du  \notag\\
   =&\, \alpha^2(t-s)^2 t^{2\alpha-\gamma-1} \int_0^{\infty} (1+v)^{2\alpha-2}v^{-\gamma}dv \\
   =&\, c_{2,1} \frac{  |t-s|^2}{t^{2-2H}}.
   \end{split}
     \end{equation*} 
 Similarly, we can prove that 
  \begin{align*} 
   \mathbb E\left[\left(Y(t)-Y(s) \right)^2\right]\le c_{2,1}\frac{ |t-s|^2}{s^{2-2H}}.
     \end{align*} 
     This proves \eqref{Eq: Ymoment1}.      
     \end{proof}

By Lemma \ref{lem Y moment},  the Gaussian property of $Y$, and the Kolmogorov continuity theorem 
(see, e.g.,  \cite[Theorem C.6]{K}),  we know that, for any $\varepsilon>0$,  $Y$ has a modification that is 
H\"older continuous with index $1-\varepsilon$ on any   interval $[a, b]$ with $0<a<b$. We will apply this fact in Section 7 
to prove Theorems \ref{thm uniform},  \ref{thm LIL0}-\ref{thm tangent}, from the results on the generalized Riemann-Liouville FBM $Z$.

In the following, we prove the differentiability of $Y$ by using the argument in the proof  of Lemma 3.6 in \cite{K}.  
 \begin{proposition}\label{prop Y}
 Assume $-1/2+\gamma/2<\alpha<1/2+\gamma/2$. The Gaussian process $Y=\{Y(t)\}_{t\ge0}$ 
 has a modification (still denoted by $Y$) that is infinitely  differentiable in $(0,\infty)$.  
 \end{proposition}
 \begin{proof} 

For any $t>0$, define 
   \begin{equation}\label{Eq:Y'}
   Y'(t):=\alpha \int_{-\infty}^0 (t-u)^{\alpha-1}(-u)^{-\frac{\gamma}{2}}B(du).
   \end{equation}
The integrand is in $L^2((-\infty,0);\mathbb R)$, and hence $\{Y'(t)\}_{ t>0}$ is a well-defined mean-zero Gaussian process. 
For every $s, t\in [a,b]\subset(0,\infty)$ with $s<t$,  applying Taylor's formula \eqref{eq Taylor1}, we have 
\begin{equation}
\begin{split}
 &\mathbb E\left[ \left| Y'(t)- Y'(s)    \right|^2  \right] \\
  =\,& \alpha^2 \int_0^{\infty} \left| (t+u)^{\alpha-1}-(s+u)^{\alpha-1}  \right|^2 u^{-\gamma}du \\ 
    \le \,&   \alpha^2(\alpha-1)^2 \left[ \int_0^1 s^{2\alpha-4} u^{-\gamma}du+\int_1^{\infty} u^{2\alpha-4-\gamma}du \right]\cdot |t-s|^2  \\ 
  \le \,&\alpha^2(\alpha-1)^2 \left(\frac{a^{2\alpha-4}}{1-\gamma}-\frac{1}{2\alpha-3-\gamma}   \right) |t-s|^2.
\end{split}
\end{equation}
This, together with the Kolmogorov continuity theorem and the arbitrariness of $a$ and $b$,    
implies that $Y'$ is continuous in  $(0,\infty)$ up to a modification.
      
Assume that $\phi\in C_c^{\infty}((0,\infty))$ (the space of all  infinitely differentiable functions with compact supports). 
By the stochastic Fubini theorem \cite[Corollary 2.9]{K}  and the formula of integration by parts, we know a.s., 
\begin{align} 
\begin{split}
\int_0^{\infty} Y'(t)\phi(t)dt&=\int_{-\infty}^0 B(du)\int_0^{\infty}\frac{\partial}{\partial t}\left[((t-u)^{\alpha}-(-u)^{\alpha})
 u^{-\frac{\gamma}{2}}   \right] \phi(t)dt \\
& = -\int_{-\infty}^0 B(du)\int_0^{\infty}   \left[((t-u)^{\alpha}-(-u)^{\alpha})u^{-\frac{\gamma}{2}}   \right]  \frac{d}{dt}\phi(t)dt.
 \end{split}
 \end{align}
Applying the stochastic Fubini theorem again, we have 
   \begin{align}\label{eq diffe}
   \int_0^{\infty} Y'(t) \phi(t)dt=-\int_0^{\infty} Y(t)\frac{d}{dt}\phi(t)dt ,\ \ \  \ \ \ \text{ a.s.}
   \end{align}
This means that  $Y'(t)$ is the weak derivative of $Y(t)$ for all $t>0$. Since $Y'$ is continuous in $(0,\infty)$, 
\eqref{eq diffe} shows that $Y'$ is in fact almost surely the ordinary derivative of $Y$ in $(0,\infty)$. By induction, 
we know that $Y$ is infinitely differentiable in $(0,\infty)$.  
   \end{proof}

\section{Moment estimates for the increments and one-sided SLND  of $Z$}
   
Consider the generalized Riemann-Liouville FBM $Z= \{Z(t)\}_{t \ge 0}$ with indices $\alpha$ and $\gamma$ 
defined by 
\begin{align}\label{eq Z}
Z(t) = \int_0^t  (t-u)^{\alpha}  u^{-\gamma/2} B(du).
\end{align}
This Gaussian process is well defined if the constants $\alpha$ and $\gamma$ satisfy $\alpha > - \frac 1 2$ 
and $\gamma < 1$ and is self-similar with index $H = \alpha- \frac{\gamma}{2}+\frac1 2$. Notice that $H\le 0$ if
$-\frac12 < \alpha \le -\frac12+\frac{\gamma}2$  and 
$H > 0$ if $\alpha  > -\frac12+\frac{\gamma}2$.

In this section, we derive optimal estimates on the moment of the increments,  prove the existence of the tangent process
and establish the one-sided 
strong local nondeterminism for $Z$. These properties are useful for studying sample properties of $Z$.
  
 \subsection{Moment estimates}
In the following, Lemmas  \ref{Lem:VarioZ},  \ref{Lem:VarioZ2} and \ref{lem expect} provide 
optimal estimates on  $\mathbb E\left[ (Z(t)-Z(s))^2 \right]$. These estimates are essential for 
establishing sharp sample path properties of $Z$. 
Notice that the upper bounds in (i) and (ii) in Lemma \ref{Lem:VarioZ} below are the same (up to a 
constant factor) when $\alpha <1/2$. We will use these bounds for estimating the small ball 
probability and the uniform modulus of continuity in Sections 4 and 7.

\begin{lemma} \label{Lem:VarioZ}
Assume  $\alpha\in(-1/2\,, 1/2]$ and $\gamma \in [0, 1)$. The following statements hold:
\begin{itemize}
\item[(i).]\,    If  $0 < s < t <\infty$ and $0 < s \le 2(t-s)$, then 
\begin{equation}\label{Eq:expect10}
 c_{3,1}   \frac {|t-s|^{2\alpha+1}} {t^{\gamma} }\le \mathbb E\left[ (Z(t)-Z(s))^2 \right]    
\le   c_{3,2}  \frac {|t-s|^{2\alpha+1}} {s^{\gamma} },
 \end{equation} 
 here  $c_{3,1}=1/(2\alpha+1)$ and $c_{3,2}= 2/(1-\gamma)+1/(2\alpha+1) +\mathcal B(2\alpha+1, 1-\gamma) 2^{ 2\alpha+1} $. 
 \item[(ii).]\, If $0 < s < t < \infty$ and $s > 2(t-s)$, then 
\begin{equation}\label{Eq:expect1a}
\mathbb E\left[ (Z(t)-Z(s))^2 \right] 
 \asymp \left\{\begin{array}{ll}
 \frac{|t-s|^{2\alpha+1}} {s^{\gamma} } , \ \quad &\hbox{ if }\  \alpha < 1/2,\\ 
 \frac{(t-s)^2} {s^{\gamma}} \left(1+\ln \big| \frac s {t-s}\big|\right),  \ \quad &\hbox{ if } \ \alpha=1/2.
\end{array}
\right.
 \end{equation} 
Here and below,  for two real-valued functions $f$ and $g$ defined on a set $I$, the notation $f \asymp g$ 
means that 
$$c \le f(x)/g(x) \le c'  \ \ \hbox{ for all }\,x \in I, $$ 
for some positive and finite constants $c$ and $c'$ which may depend on $f,\, g$ and $I$.
\end{itemize}
\end{lemma}
\begin{proof} For any $0< s<t<\infty$,  we have
\begin{equation}\label{eq Zt-Zs1a}
\begin{split}
  &\mathbb E\left[ (Z(t)-Z(s))^2 \right]  \\
      = &\, \int_0^s\left((t-u)^{\alpha}-(s-u)^{\alpha}   \right)^2 u^{-\gamma}du+ \int_s^t(t-u)^{2\alpha}u^{-\gamma}du\\
    =&:\, I_1 + I_2.
 \end{split}
  \end{equation}  
  
 To bound the integral $I_1$, we make a change of variable with $u=s-(t-s)v$ to obtain
 \begin{equation} \label{Eq:I1}
 I_1  = (t-s)^{2\alpha-\gamma+1}\int_0^{\frac{s}{t-s}}\left[(1+v)^{\alpha}-v^{\alpha} \right]^2 
 \Big(\frac{s}{t-s}-v\Big)^{-\gamma}dv.
 \end{equation}  
 In order to estimate $I_1$ in Case  (i), we distinguish the two cases  $\alpha\in[0, 1/2]$ and $\alpha\in(- 1/2, 0)$. 
 
When  $\alpha\in[0, 1/2]$, we use the fact that   
 $3^\alpha - 2^\alpha \le (1+v)^{\alpha}-v^{\alpha} \le 1$ for all $v \in [0, 2]$ to derive 
 \begin{equation}\label{Eq:I1b}
 \begin{split}
 I_1\le &\,  \big(t-s\big)^{2\alpha-\gamma+1} \int_0^{\frac{s}{t-s}} \Big(\frac{s}{t-s}-v\Big)^{-\gamma}dv\\
 =&\,\frac {s^{1-\gamma}} {1 - \gamma}  \big(t-s\big)^{2\alpha}\\
 \le &\, \frac {2} { 1 - \gamma  }\frac{\big(t-s\big)^{2\alpha +1}}{s^{\gamma}} .
\end{split}
\end{equation}  

When $\alpha\in(- 1/2, 0)$, we use the fact that   $0< v^{\alpha} -(1+v)^{\alpha} <v^{\alpha} $ for all $v \in [0, 2]$ 
to derive 
 \begin{equation}\label{Eq:I1cc}
 \begin{split}
 I_1\le &\,  \big(t-s\big)^{2\alpha-\gamma+1} \int_0^{\frac{s}{t-s}}v^{2 \alpha} \Big(\frac{s}{t-s}-v\Big)^{-\gamma}dv\\
 =&\,s^{2 \alpha-\gamma+1} \int_0^1 v^{2\alpha} (1 - v)^ {- \gamma} dv\\   
 \le &\, \mathcal B(2\alpha+1, 1-\gamma) 2^{ 2\alpha+1}  \frac{\big(t-s\big)^{2\alpha +1}}{s^{\gamma}} .
\end{split}
\end{equation}  
In the above, we made a change of variable and the assumption that $s \le 2(t-s)$.

For the integral $I_2$ in \eqref{eq Zt-Zs1a}, by the change of variable $u = s + (t -s)v$, we have
\begin{equation}\label{Eq:I2a}
 \begin{split}
 I_{2} & = (t-s)^{2 \alpha   - \gamma+1} \int_0^1(1 - v)^{2 \alpha} \Big(\frac{s}{t-s} + v\Big)^{-\gamma} \,dv\\
 &\le  (t-s)^{2 \alpha   - \gamma+1} \Big(\frac{s}{t-s}\Big)^{-\gamma}\int_0^1(1 - v)^{2 \alpha} dv\\
 &= \frac{1}{2\alpha+1}\frac{(t-s)^{2 \alpha + 1}}{ s^{\gamma}}\, .
 \end{split}
\end{equation}
On the other hand, in Case (i), $\big(\frac{s}{t-s} + v\big)^{-\gamma} \ge \big(\frac{t}{t-s} \big)^{-\gamma}$ for 
all $v \in [0, 1]$. Hence,  
\begin{equation}\label{Eq:I2b}
 \begin{split}
 I_{2}  \ge  (t-s)^{2 \alpha   - \gamma+1} \Big(\frac{t}{t-s}\Big)^{-\gamma}\int_0^1(1 - v)^{2 \alpha} dv
= \frac{1}{2\alpha+1}\frac{(t-s)^{2 \alpha + 1}}{ t^{\gamma}}\, .
 \end{split}
\end{equation}
Consequently, the lower bound in (\ref{Eq:expect10}) follows from \eqref{eq Zt-Zs1a} and \eqref{Eq:I2b}, and 
the upper bound in (\ref{Eq:expect10}) follows from \eqref{eq Zt-Zs1a},  \eqref{Eq:I1b}, \eqref{Eq:I1cc}, and 
\eqref{Eq:I2a}.

Now we consider Case (ii).  Since  $s > 2(t-s)$, we write   
   \begin{equation}
 \begin{split}\label{Eq:I1c}
 I_1 &= \big(t-s \big)^{2\alpha-\gamma+1}\Bigg(\int_0^1 \left[(1+v)^{\alpha}-v^{\alpha} \right]^2
 \Big(\frac{s}{t-s}-v\Big)^{-\gamma}dv  \\
     & \ \ \ \ \ \ \  \quad \ \ \ \ \ \ \qquad \ \ \  +\int_1^{\frac{s}{t-s}}\left[(1+v)^{\alpha}-v^{\alpha} \right]^2
     \Big(\frac{s}{t-s}-v\Big)^{-\gamma}dv\Bigg)  \\
    &=: \big(t-s \big)^{2\alpha-\gamma+1}\big(I_{1,1}+I_{1,2} \big).
    \end{split}
 \end{equation}  
We will see that the main term is  the integral $I_{1,2}$. For estimating the integral $I_{1,1}$, again 
we distinguish the two cases  $\alpha\in[0, 1/2]$ and $\alpha\in(- 1/2, 0)$. 

When $\alpha\in[0, 1/2]$, we use the facts that  
$2^\alpha - 1 \le (1+v)^{\alpha}-v^{\alpha} \le 1$ for all   $ v \in [0, 1]$
and
\begin{equation}\label{Eq:I12x}
 \frac{s}{2(t-s)} \le \frac{s}{t-s}-v \le \frac{s}{t-s}, \quad \forall \, v \in [0, 1],
\end{equation}
 to derive that 
 \begin{equation}\label{Eq:I11}
 \begin{split}
 \left(2^{\alpha}-1\right)^2 \Big(\frac{t-s}{s}\Big)^{\gamma}\le I_{1,1} \le  2^\gamma \Big(\frac{t-s}{s}\Big)^{\gamma}.
 \end{split}
\end{equation}

When $\alpha\in(- 1/2, 0)$, we use the fact that $(1-2^\alpha) v^\alpha \le v^{\alpha} - (1+v)^{\alpha} 
\le v^{\alpha}$ for $v \in [0, 1]$ and (\ref{Eq:I12x}) to get 
\begin{equation}\label{Eq:I11b}
 \begin{split}
 \frac{\left(1-2^{\alpha}\right)^2}{2 \alpha +1}  \Big(\frac{t-s}{s}\Big)^{\gamma}\le I_{1,1} \le   \frac{2^\gamma} {2 \alpha +1} 
 \Big(\frac{t-s}{s}\Big)^{\gamma}.
 \end{split}
\end{equation}

Next we estimate the integral $I_{1,2}$. For $\alpha\in(- 1/2, 1/2]$,  we use the inequality $|(1+v)^{\alpha}-v^{\alpha}| 
\asymp v^{\alpha - 1}$ for all $v \in [1, \infty)$ to derive
 \begin{equation}\label{Eq:I12}
 \begin{split}
 I_{1,2} &\asymp \int_1^{\frac{s}{t-s}} v^{2(\alpha - 1)} \Big(\frac{s}{t-s}-v\Big)^{-\gamma}\, dv\\
 &=  \Big(\frac{s}{t-s}\Big)^{2 \alpha - 1 - \gamma}  \int^1_{\frac{t-s}{s}} w^{2(\alpha - 1)} \left(1-w\right)^{-\gamma}\, dw,
 \end{split}
\end{equation}
where the above equality is obtained by the change of variable $v = \frac{s}{t-s} w$. By splitting the last 
integral over the intervals $[\frac{t-s}{s}, \frac 3 4]$ and $[\frac 3 4, 1]$, we have
 \begin{equation}\label{Eq:I12b}
 \begin{split}
 I_{1,2} &\asymp \Big(\frac{s}{t-s}\Big)^{2 \alpha - 1 - \gamma} \bigg( \int^{\frac 3 4}_{\frac{t-s}{s}} w^{2(\alpha - 1)}   dw  
 +  \int^1_{\frac 3 4} w^{2(\alpha - 1)} \left(1-w\right)^{-\gamma}\, dw\bigg)\\
 &\asymp \left\{\begin{array}{ll}
 \big(\frac{t-s}{s}\big)^{\gamma}, \qquad &\hbox{ if } -1/2 <\alpha < 1/2,\\
  \big(\frac{t-s}{s}\big)^{\gamma}\left(1+ \ln \big| \frac s {t-s}\big|\right),  &\hbox{ if } \alpha =1/2.
 \end{array}\right. 
 \end{split}
\end{equation}
Combining (\ref{Eq:I1c})-(\ref{Eq:I12b}) yields that in Case (ii),
 \begin{equation} \label{Eq:I1d}
 I_1 \asymp  \left\{\begin{array}{ll}
 \frac{(t-s)^{2 \alpha + 1}} {s^{\gamma}}, \qquad &\hbox{ if } -1/2 < \alpha < 1/2,\\
  \frac{(t-s)^2} {s^{\gamma}}\left(1+ \ln \big|\frac s {t-s}\big|\right),  &\hbox{ if } \alpha =1/2.
 \end{array}\right. 
 \end{equation} 
 It follows from this and  \eqref{eq Zt-Zs1a} that 
 \begin{equation}\label{Eq:Zt-Zs1b}
\mathbb E\left[ (Z(t)-Z(s))^2 \right] 
 \ge c_{3,3}\, \left\{\begin{array}{ll}
 \frac{|t-s|^{2\alpha+1}}{s^{\gamma}}, \ \quad &\hbox{ if } -1/2 < \alpha < 1/2,\\ 
 \frac{(t-s)^2} {s^{\gamma}} \left(1+\ln \big| \frac s {t-s}\big|\right),  \ \quad &\hbox{ if }  \alpha=1/2,
\end{array}
\right.
 \end{equation} 
 where $c_{3,3}> 0$ is a finite constant.

On the other hand, it follows from \eqref{eq Zt-Zs1a},  \eqref{Eq:I2a} and (\ref{Eq:I1d}) that
\begin{equation}\label{Eq:Zt-Zs1c}
\mathbb E\left[ (Z(t)-Z(s))^2 \right] 
 \le c_{3,4}\, \left\{\begin{array}{ll}
  \frac{|t-s|^{2\alpha+1}}{s^{\gamma}}, \ \quad &\hbox{ if }\ -1/2 < \alpha < 1/2,\\ 
 \frac{|t-s|^2} {s^{\gamma}}\left(1+ \ln \big| \frac s {t-s}\big|\right),  \ \quad &\hbox{ if } \ \alpha=1/2,
\end{array}
\right.
 \end{equation} 
where $c_{3,4} > 0$ is a finite constant. This finishes the proof of (\ref{Eq:expect1a}).  
\end{proof}

\begin{remark} \label{Re:1}{\rm The following are some remarks about Lemma  \ref{Lem:VarioZ}
and some of its consequences. 
\begin{itemize}
\item[(i).] \,  
 It follows from Lemma  \ref{Lem:VarioZ}  that for 
any  $0 < a < b < \infty$ there exist constants $c_{3,5}, \cdots, c_{3,8}\in (0,\infty)$ such that for 
$s,t\in [a, \, b]$, we have that for  $\alpha\in(-1/2,\,  1/2)$,
\begin{equation}\label{eq expect1}
c_{3,5}|t-s|^{2\alpha+1} \le \mathbb E\left[ (Z(t)-Z(s))^2 \right]  \le c_{3,6} |t-s|^{2\alpha+1};
\end{equation}  
 for  $\alpha=1/2$,
\begin{equation}\label{eq expect12}  
\begin{split}
c_{3,7}|t-s|^2  \left(1+ \big|\ln |t-s| \big|\right)\le &\, \mathbb E\left[ (Z(t)-Z(s))^2 \right] \\
& \le c_{3,8} |t-s|^2  \left(1+ \big|\ln |t-s| \big|\right). 
\end{split}
 \end{equation} 
Consequently,  the process $Z$ has a modification that is uniformly H\"older continuous 
on $[a,\, b]$ of order $\alpha+1/2-\varepsilon$ for all $\ep>0$.  In the proof of Theorem \ref{thm uniform} below, 
we will establish an exact uniform modulus of continuity of $Z$ on any interval $[a, b]$ for $0 < a < b < \infty.$

\item[(ii).]\, By Lemma \ref{lem Y moment}, we see that  for  $\alpha\in(-1/2 +\gamma/2,\,  1/2]$,  the inequalities 
\eqref{eq expect1} and \eqref{eq expect12} hold for GFBM $X$.  These inequalities can be applied to determine the 
fractal dimensions of random sets (e.g. range, graph, level set, etc) generated by the sample path of $X$. For 
example, we can derive by using standard covering and capacity methods (cf. \cite{Fal14, Kahane85, Xiao2009}) 
the following Hausdorff dimension results: for any $T > 0$,
\begin{equation}\label{Eq:Gr}
\dim_{\rm H} {\rm Gr} X([0, T])  =  2 - \Big(\frac 1 2 + \frac \alpha 2\Big) = \frac{3 - \alpha} 2, \quad \hbox{ a.s.},
\end{equation}
where ${\rm Gr} X([0, T])= \{(t, X(t)): t \in [0, T]\}$ is the graph set of $X$, and for every $x \in \mathbb R$ we have
\begin{equation}\label{Eq:Level}
\dim_{\rm H}  X^{-1}(x) =  1 - \Big(\frac 1 2 + \frac \alpha 2\Big) = \frac{1 - \alpha} 2   
\end{equation}
with positive probability, where $X^{-1}(x) = \{t \ge 0: X(t) = x \}$ is the level set of $X$. We remark that, due to the 
$\sigma$-stability of Hausdorff dimension $\dim_{\rm H}$ (\cite{Fal14}),  the asymptotic behavior of $X$ at $t = 0$ 
(hence the self-similarity index $H$) has no effect on (\ref{Eq:Gr}) and \eqref{Eq:Level}.  Later on, we will indicate 
how more precise results than (\ref{Eq:Gr}) and (\ref{Eq:Level}) can be established for $X$; see Remark \ref{Re:refine} 
below. 
\item[(iii).] \, Let $\xi = \{\xi(t)\}_{t\ge 0}$ be a  centered Gaussian process. If there exists an even, non-negative, 
and non-decreasing function $\varphi(h)$ satisfying $\lim_{h\downarrow 0}h/\varphi(h)=0$ and
$$
\mathbb E[(\xi(t+h)-\xi(t))^2]\ge \varphi(h)^2, \ \ \  \ t\ge0,h\in (0,1),
$$
then by using the argument in Yeh \cite{Yeh},  one can prove   that  the sample functions of $\xi$ are 
nowhere differentiable with probability one. See also \cite{KK1971}. Thus, if  $-1/2<\alpha\le 1/2$, then 
\eqref{eq expect1}  and \eqref{eq expect12}  imply that the sample paths  of the generalized Riemann-Liouville 
FBM $Z$ are nowhere differentiable in $(0,\infty)$ with probability one.
\end{itemize}
}
\end{remark}

The next lemma will be used for studying the tangent processes of $Z$ in Proposition
 \ref{Prop tangent Z}.

 \begin{lemma}\label{Lem:VarioZ2}
 Assume  $\alpha\in(-1/2,  1/2)$ and $\gamma \in [0, 1)$.   
Then for every fixed $t\in (0,\infty)$, 
  \begin{equation}\label{eq lim increm}
 \begin{split}
 \lim_{s\rightarrow t}\frac{ \mathbb E\left[ (Z(t)-Z(s))^2 \right] }{|t-s|^{2\alpha+1}}
 =    c_{3,9}  t^{-\gamma},
 \end{split}
 \end{equation}  
 where  $c_{3,9}= \frac{1}{2\alpha+1}+\int_0^{\infty} \big[(1+v)^{\alpha}-v^{\alpha}\big]^2dv$.
 \end{lemma}  
      
\begin{proof}  
For simplicity, we consider the limit in \eqref{eq lim increm} as $s\uparrow t$ only. For $0< s<t<\infty$, 
 recall from \eqref{eq Zt-Zs1a} the decomposition of $\mathbb E\left[ (Z(t)-Z(s))^2 \right]$.

The term $I_2$ is easier and we handle it first. Since $u^{-\gamma}\in [t^{-\gamma}, s^{-\gamma}]$ for any $u\in[s,t]$,  
we obtain that 
 \begin{align}\label{eq moment st}
 \frac{1}{  2\alpha+1}\frac{ |t-s|^{2\alpha+1}}{t^{\gamma}}\le   I_2\le  \frac{1}{ 2\alpha+1 }\frac{|t-s|^{2\alpha+1}}{s^{\gamma}}.
 \end{align} 
 This implies that 
 \begin{equation}\label{eq lim increm2}
 \lim_{s\uparrow t}\frac{ I_2  }{|t-s|^{2\alpha+1}}  =  \frac{1} {2\alpha+1} t^{-\gamma}.
 \end{equation}

For the integral $I_1$, we use (\ref {Eq:I1}) and write it as 
\begin{equation}\label{eq moment I}
\begin{split}
  I_1=&\, \int_0^s\left[(t-u)^{\alpha}-(s-u)^{\alpha} \right]^2u^{-\gamma}du \\
    =&\, (t-s)^{2\alpha+1}\int_0^{\frac{s}{t-s}}\left[(1+v)^{\alpha}-v^{\alpha} \right]^2\left(s-v(t-s) \right)^{-\gamma}dv\\
    =&(t-s)^{2\alpha+1}  \int_0^{\frac{s}{2(t-s)}} \left[(1+v)^{\alpha}-v^{\alpha} \right]^2\left(s -v(t-s)\right)^{-\gamma}dv \\
 & \ \ \ \ \ \ \ \ \  \ \ \ \ + (t-s)^{2\alpha-\gamma+1} \int_{\frac{s}{2(t-s)}}^{\frac{s}{t-s}} \left[(1+v)^{\alpha}-v^{\alpha} \right]^2
 \left(\frac{s}{t-s}-v\right)^{-\gamma}dv \\
=:&\,  J_{1,1}+J_{1,2}.\\
\end{split}
\end{equation}

Notice that, for every $v \in \big(0, s/[2(t-s)] \big)$, we have $s^{-\gamma} \le \left(s-v(t-s) \right)^{-\gamma} 
\le 2^\gamma s^{-\gamma}$. 
  Thus the dominated convergence theorem gives
 \begin{equation}\label{eq lim increm3}
 \lim_{s\uparrow t}\frac{ J_{1,1}  }{|t-s|^{2\alpha+1}}  =    t^{-\gamma} \int_0^\infty\left[(1+v)^{\alpha}-v^{\alpha} \right]^2dv.
 \end{equation}    
    
On the other hand, we use the inequality $|(1+v)^{\alpha}-v^{\alpha}| \le |\alpha| v^{\alpha - 1}$ for all $v \ge1$ to obtain
 \begin{equation*}\label{eq moment I30}
 \begin{split}
J_{1,2}=&\,(t-s)^{2\alpha+1} \int_{\frac{s}{2(t-s)}}^{\frac{s}{t-s}} \left[(1+v)^{\alpha}-v^{\alpha} \right]^2\left(s-v(t-s)\right)^{-\gamma}dv  \\
 \le &\,(t-s)^{2\alpha+1}  \int_{\frac{s}{2(t-s)}}^{\frac{s}{t-s}} \alpha^2 v^{2\alpha-2}\left(s -v(t-s)\right)^{-\gamma}dv   \\
 \le &\,\alpha^2 (t-s)^{2\alpha+1}   \left(\frac{s}{2(t-s)} \right)^{2\alpha-2} \int_{\frac{s}{2(t-s)}}^{\frac{s}{t-s}}\left(s -v(t-s)\right)^{-\gamma}dv   \\
= &\,  \frac{\alpha^2}{1-\gamma}  \left(\frac{s}{2}\right)^{2\alpha-1-\gamma}\, (t-s)^{2}.
 \end{split}
 \end{equation*}
It follows that
\begin{equation}\label{eq lim increm3b}
 \lim_{s\uparrow t}\frac{ J_{1,2}  }{|t-s|^{2\alpha+1}}  =  0.
 \end{equation}   Therefore, \eqref{eq lim increm} follows from \eqref{eq lim increm2}, \eqref{eq moment I}, \eqref{eq lim increm3} and 
\eqref{eq lim increm3b}. The proof is complete.
\end{proof}

The following lemma deals with the case when $\alpha > 1/2$ and provides estimates on the second 
moments of the increments of $Z(t)$ and its mean-square  derivative $Z'(t)$. The latter estimate 
allows us to show that $\{Z(t)\}_{ t \ge 0}$ has a modification  whose sample functions are continuously 
differentiable in $(0,\infty)$.   

For simplicity, we only consider the case when $\alpha \in (1/2,\, 3/2)$ and $s, t$ stay away from the origin. 
This is sufficient for our study of the sample path properties of  GFBM $X$.

\begin{lemma}\label{lem expect} 
Assume    $\alpha \in(1/2,\, 3/2)$ and $\gamma \in [0, 1)$. For any $0<a<b<\infty$,  it holds that for any 
$s,t\in [a,\,b]$,
 \begin{equation}\label{eq expect2}
c_{3,10}|t-s|^{2}\le  \mathbb E\left[ (Z(t)-Z(s))^2 \right]\le c_{3,11}|t-s|^{2},
\end{equation}  
here $c_{3,10}=\frac{\alpha^2}{1-\gamma}  b^{2\alpha-2}a^{1-\gamma} $ and $c_{3,11}=\alpha^2
a^{2\alpha-1-\gamma}\mathcal B(2\alpha-1, 1-\gamma)$. 

The process $\{Z(t)\}_{t \in [a,b]}$ has a modification, which is still denoted by $Z$, such that its derivative process
$\{Z'(t)\}_{t \in [a, b]}$ is continuous almost surely. Furthermore,  there exist   constants   
 $c_{3,12}, c_{3,13}\in (0,\infty)$ such that for any $s, t\in [a,b]$,
\begin{equation}\label{eq expect3}
c_{3,12}|t-s|^{2\alpha-1}\le  \mathbb E\left[ (Z'(t)-Z'(s))^2 \right]\le c_{3,13}|t-s|^{2\alpha-1}.
\end{equation}
 
\end{lemma}
   
\begin{proof}  The proof of  \eqref{eq expect2} is similar to that of Lemma \ref{Lem:VarioZ}. Here, we only 
 prove the lower bound in \eqref{eq expect2} and the existence of a modification of $Z$ 
whose sample functions are continuously differentiable on $(0, \infty)$ almost surely.   

For any $s,t\in [a,b]$ with $s<t$,  by Taylor's formula \eqref{eq Taylor1} and  \eqref{eq Zt-Zs1a},  we have 
\begin{equation*}
\begin{split}
\mathbb E\left[(Z(t)-Z(s))^2 \right]\ge &\int_0^s\left((t-u)^{\alpha}-(s-u)^{\alpha}   \right)^2 u^{-\gamma}du \\ 
\ge&\,\alpha^2(t-s)^2 \int_0^s (t-u)^{2\alpha-2} u^{-\gamma}du \\
\ge &\,\alpha^2 (t-s)^2\,  t^{2\alpha-2}\frac{  s^{1-\gamma}}{1-\gamma} \\
\ge &\,    \frac{\alpha^2}{1-\gamma}  b^{2\alpha-2}a^{1-\gamma}  (t-s)^2.
\end{split}
\end{equation*}
Thus, the lower bound in \eqref{eq expect2} holds.
 
For any $t\ge 0$, define 
\begin{align}\label{eq Z 1}
Z'(t):=\alpha \int_0^{t} (t-u)^{\alpha-1}u^{-\frac{\gamma}{2}}B(du),
 \end{align}
with $Z'(0)=0$. Notice that, since $\alpha \in (1/2, \, 3/2)$, the process $\left\{Z'(t)\right\}_{ t\ge 0}$ is a generalized 
Riemann-Liouville FBM with  indices $\alpha-1$ and $\gamma$.  It is self-similar with index $\tilde{H} = \alpha -1/2 
-\gamma/2 $. Hence,  by \eqref{eq expect1},  we see that  \eqref{eq expect3} holds. By the Kolmogorov 
continuity theorem (see, e.g.,  \cite[Theorem C.6]{K}), the Gaussian property of $Z'$ and the arbitrariness of $a$ and $b$, 
we know that  $Z'$ has a modification (still denoted by $Z'$) that is  H\"older continuous in  $(0,\infty)$ with index 
$\alpha-1/2-\varepsilon$ for any $\varepsilon>0$. With this modification, we define a Gaussian process $\widetilde{Z}
=\big\{\widetilde{Z}(t)\big\}_{t \ge 0}$ by
\[ 
\widetilde{Z}(t) = \int_0^t Z'(s) ds, \quad \forall\, t \ge 0.
\] 
Then, by the stochastic Fubini theorem  \cite[Corollary 2.9]{K}, we derive that for every $t \ge 0$,
\[
\widetilde{Z}(t) =\alpha \int_0^t  \bigg(\int_u^t (s-u)^{\alpha - 1} ds\bigg)u^{-\frac \gamma 2}\, B(du) = Z(t), \quad \hbox{ a.s.}
\]
Hence $\widetilde{Z}$ is modification of the generalized Riemann-Liouville FBM $Z=\{Z(t)\}_{t \ge 0}$ and 
the sample function of  $\widetilde{Z}$ is a.s. continuously differentiable in $(0,\, \infty)$.   
The proof is complete.
\end{proof}

\subsection{The tangent process of $Z$}\label{Sec: Tangent}
   
Let $C(\mathbb R_+, \mathbb R)$ be the space of all continuous functions from $\mathbb R_+$ to $\mathbb R$, 
endowed with the locally uniform convergence topology.  We say  that a family of stochastic processes
$V_n = \{V_n(\tau)\}_{\tau \ge 0}$  converges weakly (or in distribution) to $V = \{V(\tau)\}_{\tau \ge 0}$ in $C(\mathbb R_+, \mathbb R)$,  
if $\mathbb E[f(V_n)]\rightarrow\mathbb E[f(V)]$   for  every bounded, continuous function $f: C(\mathbb R_+, \mathbb R)
\rightarrow\mathbb R$ (cf. \cite{Bil}).

 \begin{proposition}\label{Prop tangent Z} Assume $\alpha\in(-1/2, 1/2)$ and $\gamma \in [0, 1)$.   For any $t,\,
  u>0$, let
  \begin{equation*}\label{eq VZ}
V_Z^{t,u}(\tau) :=\frac{ Z(t+u\tau)-Z(t)}{u^{\alpha+1/2}}, \ \ \    \tau\ge0.
 \end{equation*} 
 Then, as $u\rightarrow0+$, the process $\big\{V_Z^{t,u}(\tau)\big\}_{\tau\ge0}$  converges in distribution to 
 $\big\{c_{3,9}^{1/2} t^{-\gamma/2} B^{\alpha+1/2}(\tau)\big\}_{\tau\ge 0}$ in $C(\mathbb R_+,\mathbb R)$,   
 where   $c_{3,9}=\frac{1}{2\alpha+1}+\int_0^{\infty} \big[(1+v)^{\alpha}-v^{\alpha}\big]^2dv$ is the constant in  
Lemma \ref{Lem:VarioZ2}. 
        \end{proposition} 
   \begin{proof} By the self-similarity of $Z$ and \eqref{eq lim increm}, we know that for any $t, s>0$,
   \begin{equation}\label{eq sim t-s}
   \begin{split} 
   &\mathbb E\big[Z(t)Z(s)\big]\\
   =&\,\frac{1}{2} \mathbb E\big[Z(t)^2+Z(s)^2-(Z(t)-Z(s))^2\big]\\
    = &\, \frac{1}{2}\Big(\mathcal B(2\alpha+1, 1-\gamma) \left(t^{2H}+s^{2H}\right)- c_{3,9} t^{-\gamma}|t-s|^{2\alpha+1}  \Big) 
   +o(t^{-\gamma}|t-s|^{2\alpha+1}).
   \end{split}
   \end{equation}
   For any $t>0, \tau_1, \tau_2\ge0$, by  \eqref{eq sim t-s}, we have 
   \begin{align}
  \lim_{u\rightarrow0+}  \mathbb E\left[V_Z^{t,u}(\tau_1) V_Z^{t,u}(\tau_2)   \right] =\frac{c_{3,9}}{2}t^{-\gamma}
  \left(\tau_1^{2\alpha+1}+\tau_2^{2\alpha+1}-|\tau_1-\tau_2|^{2\alpha+1}  \right).
   \end{align}
   Hence,  the   Gaussian process $\{V_Z^{t, u}(\tau)\}_{\tau\ge0}$ converges to $c_{3,9}^{1/2} t^{-\gamma/2} 
   B^{\alpha+1/2}$ in finite-dimensional distributions, as $u\rightarrow0+$,  where $B^{\alpha+1/2}$ is a FBM with index $\alpha+1/2$. 
     By Lemma  \ref{Lem:VarioZ}, there exists a constant $c_{3,14}\in (0,\infty)$ such that for   all $t>0, u>0, \tau_1, \tau_2\ge 0$,
   \begin{align}
   \mathbb E\left[ \left(V_Z^{t, u}(\tau_1)- V_Z^{t, u}(\tau_2)\right)^2\right]\le c_{3,14} t^{-\gamma} \left|\tau_1-\tau_2\right|^{2\alpha+1}.
   \end{align} 
   Hence, by \cite[Proposition 4.1]{Fal02}, we know the family $\big\{V_Z^{t,u} \big\}_{u>0}$ is tight in 
   $C(\mathbb R_+, \mathbb R)$ and  then it  converges in distribution  in $C(\mathbb R_+,\mathbb R)$,   as $u\rightarrow0+$.  
   The proof is complete.     
   \end{proof}

 \subsection{One-sided  strong local nondeterminism}
   
We establish the following one-sided strong local nondeterminism (SLND, in short) for $Z$. This 
property is essential for dealing with problems that involve joint distributions of random variables 
$Z(t_1), \ldots, Z(t_n)$.  From the proof, it is clear that GFBM $X$ has the same SLND property.  

   \begin{proposition}\label{lem SLND}
   \begin{itemize}
    \item[(a)] Assume  $\alpha\in(-1/2, 1/2]$ and $\gamma \in [0, 1)$.  For any 
   constant $b>0$,  it holds that  for any $s, t\in[0,b]$ with $s<t$,
   \begin{align}\label{eq one-sided}
   \Var(Z(t)|Z(r): r\le s)\ge \frac{1}{( 2\alpha+1)b^{\gamma}} |t-s|^{2\alpha+1},
   \end{align}    where $\Var(Z(t)|Z(r): r\le s)$ denotes the conditional variance of $Z(t)$ given the $\sigma$-algebra
   $\sigma(Z(r): r\le s)$.
   \item[(b)] Assume  $ \alpha\in(1/2,\, 3/2)$ and $\gamma \in [0, 1)$. 
   For any $b>0$, it holds that for any $s, t\in[0,b]$ with $s<t$,
   \begin{align}\label{eq one-sided2}
   \Var\left(Z'(t)|Z'(r): r\le s\right)\ge \frac{1}{( 2\alpha-1)b^{\gamma}}|t-s|^{2\alpha-1}.
   \end{align}  
\end{itemize}
   \end{proposition}
   \begin{proof}    Assume that $0 \le s< t\le b$. 
   We write $Z(t)$ as 
   \begin{align*}
   Z(t)=& \int_0^s (t-u)^{\alpha}u^{-\frac{\gamma}{2}}B(du)+\int_s^t (t-u)^{\alpha}u^{-\frac{\gamma}{2}}B(du).
   \end{align*}
   The first term is measurable with respect to $\sigma\big(B(r):r\le s\big)$ and the second term is independent of $\sigma\big(B(r):r\le s\big)$. 
   Since $\sigma\big(Z(r): r\le s\big) \subseteq \sigma\big(B(r): r\le s\big)$, 
   we have
\begin{align*}
\Var(Z(t)| Z(r): r\le s)\ge  \,\Var(Z(t)| B(r): r\le s) 
=  \, \int_s^t(t-u)^{2\alpha}u^{-\gamma}du.
\end{align*}
This implies  \eqref{eq one-sided}.

The proof of \eqref{eq one-sided2} is similar. The details are omitted. The proof is complete. 
\end{proof}

\section{Lamperti's transformation of $Z$}
\label{sec:Lamperti}
Inspired by \cite{TX2007}, we   consider the centered stationary Gaussian process $U =\{U (t)\}_{t\in\mathbb R}$ 
defined through  Lamperti's transformation of $Z$:
 \begin{align}\label{Eq:UZ}
 U(t):=e^{-tH} Z(e^t)  \ \ \ \text{ for all } t\in \mathbb R.
 \end{align} 
Let   $r_U(t):=\mathbb E \big[U (0)U (t)\big]$ be 
 the covariance function of $U$.  By Bochner's theorem, $r_U$ is the Fourier transform of a finite measure $F_U$ 
 which is called the spectral measure of $U$. Notice that $r_U(t)$ is an even function and 
  \begin{align}\label{eq rt}
 r_U(t)=&\, e^{-tH}\int_0^{1\wedge e^t}(e^t-u)^{\alpha}(1-u)^{\alpha} u^{-\gamma}du \ \ \ \ \text{for all } t\in \mathbb R.
 \end{align}
 We can verify that  $r_U(t)=O(e^{- t(1-\gamma)/2})$ as $t\rightarrow \infty$. 
 It follows that $r_U(\cdot)\in L^1(\mathbb R)$. Hence the spectral measure $F_U$ has a continuous spectral 
 density function $f_U$ which can be represented as the inverse Fourier transform of $r_U(\cdot)$:
 $$
 f_U(\lambda)=\frac1\pi\int_0^{\infty}r_U(t)\cos(t\lambda)\,dt  \quad \hbox{ for all } \, \lambda \in \mathbb R.
 $$   
  It is known that $U $ has the stochastic integral representation:
   \begin{align}\label{eq U int}
   U(t)=\int_{\mathbb R} e^{i\lambda t}\, W(d\lambda)    \ \ \ \ \text{for all } t\in \mathbb R,
   \end{align}
   where $W$ is a complex Gaussian measure with control measure $F_U$. Then for any $s, t \in \mathbb R$,
   \begin{equation}\label{Eq:Uvar}
   \begin{split}
    \mathbb E \big[ \big(U (s) -U (t)\big)^2\big] &= 2 \big( r_U(0) - r_U(t-s) \big)\\
    &=2 \int_{\mathbb R} \big[1 - \cos \big((s-t) \lambda\big)\big]\, f_U(\lambda) d \lambda.
    \end{split}
    \end{equation}
   
The following lemma provides bounds for $ \mathbb E \left[ \left(U (s) -U (t)\right)^2\right]$  when $Z$ has rough 
(fractal) sample paths.
 \begin{lemma}\label{lem r}  
 Assume  $\alpha\in (-1/2, 1/2]$. Then for any $b > 0$, there exist positive constants $\varepsilon_0$, $c_{4,1}$ 
 and $c_{4,2}$ such that  for all $ s, t \in [0, b]$ with $|s-t| \le \varepsilon_0$,
  \begin{equation}\label{eq r}
  \begin{split}
 c_{4,1}|s-t|^{2 \alpha+1}\left(1+ \left|\ln |t-s| \right|\right)^\eta &\le \mathbb E \left[ \left(U (s) -U (t)\right)^2\right] \\
 &\le  
 c_{4,2}|s-t|^{2 \alpha+1}\left(1+ \left|\ln |t-s| \right|\right)^\eta,
  \end{split}
 \end{equation}
  where $\eta = 0$ if  $\alpha\in (-1/2, 1/2)$ and $\eta=1$ if $\alpha = 1/2.$
 \end{lemma}
   \begin{proof}   Since $U $ is stationary, it is sufficient to consider $\mathbb E \left[ \left(U (t) -U (0)\right)^2\right]$ for $t >0$.
    It follows from \eqref{Eq:UZ} and the elementary inequality $(x+y)^2 \le 2 (x^2 + y^2)$ that 
   \begin{equation}\label{Eq:Umin}
   \begin{split}
   \mathbb E \left[ \left(U (t) -U (0)\right)^2\right] 
   & = \mathbb E \left[ \left(  Z(e^t) -Z(1)  + (e^{-tH} -1) Z(e^t)\right)^2\right]\\
   &\le 2\left(\mathbb E \left[ \left(  Z(e^t) -Z(1) \right)^2\right] + \left(e^{-tH} -1\right)^2  \mathbb E \left[Z(e^t)^2\right]\right).
   \end{split}
   \end{equation}
This, together with  Lemma  \ref{Lem:VarioZ}, implies  that  the upper bound in \eqref{eq r} 
 holds for all $s, t \in [0, b]$  with $|s-t| \le \varepsilon_0$ for some positive constant $\varepsilon_0$. 
   
  On the other hand, the first equation in \eqref{Eq:Umin} and the inequality $(x+y)^2 \ge \frac 1 2 x^2 -  y^2$ imply
   \begin{equation}
   \begin{split}
   \mathbb E \left[ \left(U (t) -U (0)\right)^2\right]  \ge
 \frac 1 2 \mathbb E \left[ \left(  Z(e^t) -Z(1) \right)^2\right] - \left(e^{-tH} -1\right)^2  \mathbb E \left[Z(e^t)^2\right].\\
   \end{split}
   \end{equation}
 It follows from Lemma  \ref{Lem:VarioZ}   that the lower bounds in \eqref{eq r}  holds if $t>0$ is small enough, say, 
 $0 < t \le \varepsilon_0$. This completes the proof of Lemma \ref{lem r}.
\end{proof}
       
The following are truncation inequalities in Lo\'eve \cite[Page 209]{Lo} that are expressed in terms of the 
spectral density function $f_U$:  for any $u > 0$ we have
\begin{equation*}
\begin{split}
 \int_{|\lambda| < u}\ \lambda^2 f_U(\lambda) d \lambda &\le K u^2 \int_{\mathbb R} (1 - \cos (\lambda/u))
f_U(\lambda) d \lambda,\\
\int_{|\lambda| \ge u}\  f_U(\lambda)\,d \lambda &\le K  u \int_0^{
1/u} dv \int_{\mathbb R} \big(1 - \cos (v\lambda)\big) f_U(\lambda)d \lambda.
\end{split}
\end{equation*}
By these inequalities, (\ref{Eq:Uvar})  and the upper bound in Lemma \ref{lem r},  we have the 
following properties of the spectral density $f_U(\lambda)$ at the origin and infinity, respectively.
 
 \begin{lemma}\label{lem spectral}  
Assume  $\alpha\in (-1/2, 1/2]$. There exist positive constants $u_0$, 
$c_{4,3}$ and $c_{4,4}$ such that for any $u >u_0$,
 \begin{align}\label{Eq: sp1}
 \int_{|\lambda|<u}\lambda^2 f_U(\lambda) d\lambda\le c_{4,3} u^{1-2\alpha} \left(1+ |\ln u |\right)^\eta
 \end{align}
   and 
  \begin{align}\label{Eq: sp2}
 \int_{|\lambda|\ge u}  f_U(\lambda) d\lambda\le c_{4,4} u^{-(2\alpha+1)} \left(1+ |\ln u |\right)^\eta.
 \end{align}
 In the above,  $\eta = 0$ if  $\alpha\in (-1/2, 1/2)$ and $\eta=1$ if $\alpha = 1/2.$
 \end{lemma}

 \section{Small ball probabilities of $Z$}
   
By Lemma  \ref{Lem:VarioZ}, Lemma \ref{lem expect} and Proposition \ref{lem SLND}, we prove the following estimates 
on the small ball probabilities of the Gaussian process  $Z$ and its derivative $Z'$ when it exists. 
  \begin{proposition}\label{lem small} 
   \begin{itemize}
   \item[(a).]\,
 Assume $\gamma \in [0, 1)$ and $\alpha\in(-1/2+\gamma/2,\,1/2)$. There exist constants $c_{5,1}, \,
    c_{5, 2}\in(0,\infty)$ such that for all $r>0, 0<\varepsilon<1$,
   \begin{equation} \label{Eq:sb1}
   \begin{split}
  \exp\bigg(- c_{5,1}\Big(\frac{r^{H}}{\varepsilon}\Big)^{\frac 2 { 2\alpha+ 1}}\bigg)\le 
  \mathbb P\bigg\{\sup_{s\in [0,\, r]}|Z(s)|\le \varepsilon \bigg\}
  \le \exp\bigg(- c_{5,2}\Big(\frac{r^{H}}{\varepsilon}\Big)^{\frac{2}{ 2\alpha+1}}\bigg).
  \end{split}
   \end{equation} 
  \item[(b).]\,
  Assume $\gamma \in [0, 1)$ and $\alpha \in(1/2+\gamma/2, 3/2)$. 
  There exist constants $c_{5,3}, c_{5, 4}\in (0,\infty)$ such that for all $r>0, 0<\varepsilon<1$,
  \begin{equation} \label{Eq:sb2}
  \begin{split}
 \exp\bigg(- c_{5,3}\Big(\frac{r^{\tilde H}}{\varepsilon}\Big)^{\frac 2 { 2\alpha- 1}}\bigg)\le  
 \mathbb P\bigg\{\sup_{s\in [0,\, r]}|Z'(s)|\le \varepsilon \bigg\}
\le \exp\bigg(- c_{5,4}\Big(\frac{r^{\tilde H}}{\varepsilon}\Big)^{\frac{2}{ 2\alpha-1}}\bigg),
 \end{split}
  \end{equation}
  where $\tilde{H} = \alpha-\gamma/2-1/2$.
 \end{itemize} 
 \end{proposition}
 
\begin{remark}  {\rm The following are two remarks about Proposition \ref{lem small}.
 \begin{itemize}
\item Notice that the case of $\alpha = 1/2$ is excluded in Proposition \ref{lem small}.
The reason is that,   the    bounds of the one-sided SLND in Proposition \ref{lem SLND} 
     and the optimal bounds in  \eqref{Eq:expect1a} do not coincide when $\alpha = 1/2$. The method for proving  \eqref{Eq:sb1} will not be able 
to prove optimal upper and lower bounds in the case of $\alpha = 1/2$. 

\item In Part (b), we assume that the self-similarity index $\tilde{H} = \alpha-\gamma/2-1/2$ of $Z'$ 
is positive.  When $\tilde{H} \le 0$, \eqref{Eq:sb2} does not hold. In fact, by using $\mathbb E[Z'(t)^2] 
= \mathbb E[Z'(1)^2] \, t^{2\tilde{H}}$, one sees that for any $0 < \tau \le r$, 
$$\mathbb P\bigg\{\sup_{s\in [0,\, r]}|Z'(s)|\le \varepsilon \bigg\} \le \mathbb P\big\{|Z'(\tau)|\le \varepsilon \big\} = 
 \mathbb P\left\{|Z'(1)|\le \tau^{-\tilde{H}} \varepsilon  \right\}.$$
If $\tilde{H} < 0$, then the last probability goes to 0 as $\tau \to 0$. This implies that for all $r>0$ and $\varepsilon>0$
we have $\mathbb P\big\{\sup_{s\in [0,\, r]}|Z'(s)| \le \varepsilon \big\} = 0$. 

When $\tilde{H} = 0$, the self-similarity implies that the small ball probability in \eqref{Eq:sb2} does not depend 
on $r$. We will let $r \to \infty$ to show that this probability is in fact 0 for all $\varepsilon>0$. To this end, we 
consider the centered stationary Gaussian process $V= \{V(s)\}_{s \in \mathbb R}$ defined by $V(s) = Z'(e^{s})$, 
which is the Lamperti transform of $Z'$, and apply Theorem 5.2 of Pickands \cite{Pickands67}. Notice that the 
covariance function of $V$,
\[
r_V(s) = \int_0^1 (1 - u)^{\alpha-1}(e^s - u)^{\alpha-1}u^{-\gamma}\, du, \ \ \text{for } s>0. 
\]
 Then  $r_V(s) \le 2^{1-\alpha}\, e^{(\alpha - 1)s}$ for all $s $ large enough. Since $\alpha = \frac{1+\gamma}  2 < 1$, 
we have 
$\lim\limits_{s \to \infty} r_V(s) \ln s = 0$, so the condition of \cite[Theorem 5.2]{Pickands67} is 
satisfied. Hence
\[
\liminf_{t \to \infty} \bigg(\sup_{s\in [0,\, t]} V(s) - \sqrt{2 \ln t }\bigg) \ge 0, \quad \hbox{ a.s.}
\]
This implies that for all $\varepsilon>0$,
$$\mathbb P\bigg\{\sup_{s\in [0,\, r]}|Z'(s)|\le \varepsilon \bigg\} 
= \lim\limits_{t \to \infty} \mathbb P\bigg\{\sup_{s\in [0,\, t]}|Z'(s)|\le \varepsilon \bigg\} = 0.$$ 
\end{itemize}
}
 \end{remark}
 
For proving  the lower bound in \eqref{Eq:sb1}, we apply the general lower bound on the small ball probability
of Gaussian processes due to Talagrand (cf. Lemma 2.2 of \cite{Tal95}). We will make use of the following  
reformulation of Talagrand's lower bound given by Ledoux \cite[(7.11)-(7.13) on Page 257]{Ledoux}.
 
\begin{lemma} \label{Lem:Ta93}
Let $\{ Z(t)\}_{t \in S }$ be a separable, real-valued, centered Gaussian process
indexed by a bounded set $S$ with the canonical metric $d_Z(s, t) = (\mathbb E |Z(s) - Z(t)|^2)^{1/2}$.
Let $N_\varepsilon(S)$ denote the  smallest number of $d_Z$-balls of radius $\varepsilon$
needed to cover $S$. If there is a decreasing function $\psi : (0, \delta] \to (0, \infty)$ such that 
$N_\varepsilon(S) \le \psi(\varepsilon)$ for all $\varepsilon \in (0, \delta]$ and there are constants 
$K_2 \ge K_1 > 1$ such that
\begin{equation}\label{Eq:Covering}
K_1 \psi (\varepsilon) \le \psi (\varepsilon/2) \le K_2 \psi (\varepsilon)
 \end{equation}
for all $\varepsilon \in (0, \delta]$, then there is a constant $K\in (0,\infty)$ depending only
on $K_1$, $K_2$ and $d_Z$ such that for all $u \in (0, \delta)$,
\begin{equation}\label{Eq:SB1}
\mathbb P\bigg\{ \sup_{s, t \in S} |Z(s) - Z(t)| \le u \bigg\} \ge \exp \big(-K \psi(u) \big).
\end{equation}
\end{lemma}

We are ready to prove Proposition \ref{lem small}.
 \begin{proof} We will only prove \eqref{Eq:sb1}. The proof of \eqref{Eq:sb2} is the same because $Z'$ is also 
 a generalized GFBM with indices    $\gamma \in [0, 1)$ and $\alpha-1 \in (-1/2 + \gamma/2,\, 1/2)$.
 
By the self-similarity property of $Z$, we know that  \eqref{Eq:sb1} is equivalent 
 to the following statement:  there exist constants $c_{5,1}, c_{5,2}\in(0,\infty)$ such that for all $0<\varepsilon<1$,
    \begin{equation} \label{Eq:sb12}
  \exp\bigg(- c_{5,1}\Big(\frac{1}{\varepsilon}\Big)^{\frac 2 {2 \alpha+ 1}}\bigg)\le  
  \mathbb P\bigg\{\sup_{s\in [0,\, 1]}|Z(s)|\le \varepsilon \bigg\}
  \le \exp\bigg(- c_{5,2}\Big(\frac{1}{\varepsilon}\Big)^{\frac{2}{2 \alpha+1}}\bigg).
   \end{equation}
   
 In order to prove the lower bound in \eqref{Eq:sb12}, we take $S = [0, 1]$ and apply Lemma \ref{Lem:Ta93}. 
 For any $\varepsilon \in (0, \, 1)$, we construct a covering of $[0, 1]$ by 
 sub-intervals of $d_Z$-radius $\varepsilon$, which will give an upper bound for $N_\varepsilon([0, 1]).$
 
Recall that  $\mathbb E\left[ Z(t)^2 \right] = c\, t^{2H}$ for all $t \ge 0$, where $H = \alpha  - \frac{\gamma} 2 + \frac 1 2>0$.
Since constants $c$ here and those in Lemma  \ref{Lem:VarioZ}  can be absorbed by the constants 
$c_{5,1}$ and $c_{5,2}$ in \eqref{Eq:sb12}, without loss of generality we will take these constants to be 1 (otherwise 
we consider the processes obtained by dividing $Z$ by the maximum and minimum of these constants, respectively, 
and prove the upper and lower bounds in \eqref{Eq:sb12} separately.)

Let $t_0=0, t_1 = \varepsilon^{1/H}$. For  any $n \ge 2$, if $t_{n-1}$ has been defined, we define
\begin{equation}\label{def:tn}
t_n = t_{n-1} + t_{n-1}^{\frac \gamma {2\alpha + 1}} \varepsilon^{\frac 2 {2\alpha + 1}}.
\end{equation}   
It follows from Lemma  \ref{Lem:VarioZ}  that
\begin{equation} \label{Eq:tn2}
\begin{split}
\mathbb E\left[(Z(t_n) - Z(t_{n-1}))^2\right] &\le  \frac {c} {t_{n-1}^\gamma} \big| t_n - t_{n-1}\big|^{2 \alpha + 1} 
\le  \varepsilon^2.
\end{split}
\end{equation}
Hence $d_Z(t_n,\, t_{n-1}) \le  \,\varepsilon$ for all $n \ge 1$.

Since $[0, 1]$ can be covered by the intervals $[t_{n-1}, \, t_n]$ for $n = 1, 2, \ldots, L_\varepsilon$, where 
$L_\varepsilon$ is the largest integer $n$ such that $t_n \le 1$, we have $N_\varepsilon([0, 1]) \le L_\varepsilon+1
\le 2L_\varepsilon$.

In order to estimate $L_\varepsilon$, we write $t_n = a_n \varepsilon^{1/H}$ for all $n \ge 1$.
Then, by \eqref{def:tn}, we have $a_1=1$,
\begin{equation}\label{Eq:an1}
a_n = a_{n-1} + a_{n-1}^{\frac \gamma {2\alpha + 1}}, \quad \forall \, n \ge 2.
\end{equation}

Denote by $\beta = 1 - \frac \gamma {2\alpha + 1} = \frac{2H}{2\alpha + 1}$. We claim that there exist positive and finite 
constants $c_{5,5} \le  2^{-\gamma/(2H\beta)}\beta^{1/\beta}$ and $c_{5,6} \ge 1$ such that 
\begin{equation}\label{Eq:an}
c_{5,5}\, n^{1/\beta} \le a_n \le c_{5,6}\, n^{1/\beta}, \qquad \forall \, n \ge 1.
\end{equation}
We postpone the proof of \eqref{Eq:an}. Let us estimate $L_\varepsilon$ and prove the lower bound in \eqref{Eq:sb12} first. 

By \eqref{Eq:an}, we have
\begin{equation*}\label{Eq:Lep}
L_\varepsilon = \max\left\{n: a_n \le \varepsilon^{- \frac 1 H}\right\} \le c_{5,5}^{-\beta}\, \varepsilon^{- \frac \beta H} 
= c_{5,5}^{-\beta}\, \varepsilon^{- \frac 2 {2\alpha + 1}}.
\end{equation*}
This implies that for all   $\varepsilon \in (0, 1)$,
\begin{align}\label{Eq: Ncov}
N_\varepsilon([0, 1]) \le 2 c_{5,5}^{-\beta}\, \varepsilon^{- \frac 2 {2\alpha + 1}} =: \psi(\varepsilon).
\end{align}
Since the function $\psi(\varepsilon)$ satisfies \eqref{Eq:Covering} with $K_1=K_2 = 2^{ \frac 2 {2\alpha + 1}}>1,$ 
we see that the lower bound in \eqref{Eq:sb12} follows from (\ref{Eq: Ncov}) and \eqref{Eq:SB1} in Lemma \ref{Lem:Ta93}.

Now we prove \eqref{Eq:an} by using induction. Clearly  \eqref{Eq:an} holds for $n=1$. Assume that it holds 
for $n=k$.  Then for $n=k+1$, it follows from (\ref{Eq:an1}) and (\ref{Eq:an}) that 
\[
a_{n+1} \le c_{5,6}\, n^{1/\beta} + \big(c_{5,6}\, n^{1/\beta})^{\frac \gamma {2\alpha + 1}} \le c_{5,6}\, (n+1)^{1/\beta}, 
\]
where the last inequality can be checked by using the mean-value theorem and the facts that $c_{5,6}\ge 1$ 
and $0< \beta < 1$.  This verifies the upper bound in \eqref{Eq:an}. The desired lower bound for $a_{n+1}$ 
is derived similarly using the mean-value  theorem and the fact that $c_{5,5} \le  2^{-\gamma/(2H\beta)}\beta^{1/\beta}$. 
Hence  the claim \eqref{Eq:an} holds.

Next, we prove the upper bound in \eqref{Eq:sb12}. Lemma   \ref{Lem:VarioZ} and Proposition \ref{lem SLND} 
show that the conditions of Theorem 2.1 of  Monrad and Rootz\'en \cite{MR1995} are satisfied. Hence the 
upper bound in \eqref{Eq:sb1} follows from Theorem 2.1 of \cite{MR1995}. The proof is complete.
 \end{proof}   
 
Similarly to the proof of Proposition \ref{lem small}, we can prove the following estimates on the small ball 
probabilities for the increments of $Z$ and $Z'$  at  points away from the origin. We will use these estimates 
to prove Chung's LILs for $Z$ and  $Z'$. Also, notice that no extra condition of $\alpha-\gamma/2-1/2>0$ is 
 needed for (b).
 
 \begin{proposition}\label{lem small2} 
   \begin{itemize}
   \item[(a).]\, Assume   $\alpha\in(-1/2,\,1/2)$.  Then   
   there exist constants $c_{5,7}, c_{5,8}\in (0,\infty)$  such that for all $t>0, r\in \big(0,t/2\big)$ 
   and $\varepsilon\in \left(0, r^{\alpha+1/2}\right)$,
   \begin{equation} \label{Eq:sb121}\
   \begin{split}
  \exp\bigg(-c_{5,7}\, r \, t^{-\frac{\gamma}{2\alpha+1}} \varepsilon^{-\frac2{2\alpha+1}}\bigg)\le & 
 \, \mathbb P\bigg\{\sup_{|s|\le r}|Z(t+s)-Z(t)|\le \varepsilon \bigg\}\\  & \le \,
 \exp\bigg(-c_{5,8}\, r\, t^{-\frac{\gamma}{2\alpha+1}}  \varepsilon ^{-\frac{2}{2\alpha+1}}\bigg).
  \end{split}
   \end{equation} 
   \item[(b).]\,
   Assume   $\alpha\in(1/2,\, 3/2)$.  Then there 
   exist constants $c_{5,9}, c_{5,10}\in (0,\infty)$  
   such that for all $t>0, r\in \big(0,t/2\big)$ and  $\varepsilon\in \left(0,r^{\alpha-1/2}\right)$,
   \begin{equation} \label{Eq:sb221}
   \begin{split}
  \exp\bigg(-c_{5,9}\, r\, t^{-\frac{\gamma}{2\alpha-1}} \varepsilon^{-\frac2{2\alpha-1}}\bigg)\le &
 \,  \mathbb P\bigg\{\sup_{|s|\le r}|Z'(t+s)-Z'(t)|\le \varepsilon \bigg\}\\
   & \, \le \exp\bigg(-c_{5,10}\, r\, t^{-\frac{\gamma}{2\alpha-1}} \varepsilon^{-\frac{2}{2\alpha-1}} \bigg).
  \end{split}
   \end{equation}  
   \end{itemize}
   \end{proposition}
 
 \begin{proof}  (a). Assume  $\alpha\in(-1/2,\,1/2), 0<r<t/2$. Let $I(t,r) = [t-r, t+r] $. It 
 follows from Lemma \ref{Lem:VarioZ}  that
 \[
 \mathbb E\big[(Z(s) - Z(s'))^2\big] \asymp  t^{-\gamma}\,|s-s'|^{2 \alpha + 1} \quad \hbox{ for all }\ s, \,s' \in I(t,r).
 \]
 Hence, there exists a constant $c_{5,11}\in (0, \infty)$ satisfying that   for all  $0<\varepsilon< r^{\alpha+1/2}$,   
$$
N_\varepsilon(I(t,r)) \le c_{5,11}\, t^{-\frac{\gamma}{2\alpha+1}} \, \varepsilon^{-\frac 2 {2\alpha + 1}}\, r =: \psi(\varepsilon).
$$ 
Then the function $\psi(\varepsilon)$ satisfies \eqref{Eq:Covering} with $K_1=K_2 = 2^{ \frac 2 {2\alpha + 1}}>1$.
Hence the lower bound in \eqref{Eq:sb121} follows from \eqref{Eq:SB1} in Lemma \ref{Lem:Ta93}.

Next,  
Lemma \ref{Lem:VarioZ}  and Proposition \ref{lem SLND} 
show that the conditions of Theorem 2.1 of  Monrad and Rootz\'en \cite{MR1995} are satisfied. Hence the 
upper bound in \eqref{Eq:sb121} follows from Theorem 2.1 of \cite{MR1995}.

(b). As noted earlier, when $\alpha\in(1/2,\, 3/2)$ the Gaussian process $\{Z'(t)\}_{t \ge 0}$  is a generalized 
Riemann-Liouville FBM. Hence \eqref{Eq:sb221} follows from \eqref{Eq:sb121}. This finishes the proof.
 \end{proof}

\section{Chung's law of the iterated logarithm for $Z$}
   
As applications of small ball probability estimates, Monrad and
Rootz\'en \cite{MR1995}, Xiao \cite{Xiao1997}, and Li and Shao \cite{LiShao01} established
Chung's LILs for fractional Brownian motions and other strongly locally nondeterministic Gaussian
processes with stationary increments. Notice that  the generalized Riemann-Liouville FBM $Z$ does not 
have stationary increments.  Here, we will use the small ball probability estimates and the Lamperti 
transformation in the last two sections to establish Chung's LIL for $Z$ at points away from the origin when 
$\alpha\in(-1/2,\,1/2)$. The case of $\alpha = 1/2$ is open for $Z$ as well as  GFBM $X$.

\begin{proposition}\label{lem:ZLIL}
Assume $\alpha\in(-1/2,\,1/2)$.  There exists a constant $c_{6,1}\in(0,\infty)$ such that for every 
$t>0$,  
\begin{align}\label{eq ZLIL1}
\liminf_{r\rightarrow0+}\sup_{ |s|\le r} \frac{ |Z(t+s)-Z(t)|} {r^{\alpha+1/2}/(\ln \ln 1/r)^{\alpha+1/2}}
   =c_{6,1} t^{-\gamma/2}, \ \  \ \ \ \text{a.s.}
\end{align}  
\end{proposition}

For proving Proposition \ref{lem:ZLIL}, we will make use of the following zero-one law, which implies 
the existence of the limit in the left hand side of \eqref{eq ZLIL1}. Notice that the constant $c_{6,1}'$ 
in \eqref{Eq:Chung01law} can be $0$ or $\infty$.

\begin{lemma}\label{lem:01lawZ} 
Assume $\alpha\in(-1/2,\,1/2)$. There exists a constant $c_{6,1}'\in[0,\infty]$ such that for every $t > 0$,
\begin{equation}\label{Eq:Chung01law} 
 \liminf_{r\rightarrow0+}\sup_{|s|\le r}  \frac{ |Z(t+s)-Z(t)|} {r^{\alpha+1/2}/(\ln \ln 1/r)^{\alpha+1/2}}
 =c_{6,1}'t^{-\gamma/2}, \ \  \ \ \ \text{a.s.}
\end{equation}
\end{lemma}
\begin{proof} We start with the stationary Gaussian process $U = \{U(s)\}_{s \in \mathbb R}$ in \eqref{eq U int}. 
As in the proof of Lemma 2.1 in \cite{WSX}, we write for every $n \ge 1$,
\[
U_n(s) = \int_{n-1 \le |\lambda| < n} e^{i\lambda s}\, W(d\lambda).
\]
Then the Gaussian processes $U_n = \{U_n(s)\}_{s \in \mathbb R},\,  n \ge 1 $,  are independent and 
\[
U(s) = \sum_{n=1}^\infty U_n(s), 
\]
where the series is a.s. uniformly convergent on every compact interval in $\mathbb R$. As in the proof of Proposition \ref{prop Y},
we can verify that every $U_n(s)$ is a.s. continuously differentiable (or as in \cite{WSX}, it is sufficient 
to show that $U_n(s)$ is almost Lipschitz on compact intervals). 
Therefore,  for every $x\in \mathbb R, N\in\mathbb N$,
$$
 \liminf_{r\rightarrow0+}\sup_{|s|\le r}  \frac{\big|\sum_{n=1}^N \left(U_n(x+s)-U_n(s) \right)\big|} {r^{\alpha+1/2}/(\ln \ln 1/r)^{\alpha+1/2}}=0.
$$
Hence for every $x \in\mathbb R $ and $c\ge 0$, the event 
$$E_c =\bigg\{ \liminf_{r\rightarrow0+}\sup_{|s|\le r}  \frac{ |U(x+s)-U(x)|} {r^{\alpha+1/2}/(\ln \ln 1/r)^{\alpha+1/2}} \le c\bigg\}$$  
is a tail event with respect to $U_n\, (n \ge 1)$. By Kolmogorov's zero-one law, we have $\mathbb P(E_c) = 0$ or $1$. Let $c_{6,1}' =
 \inf\{c \ge 0:  \mathbb P(E_c) = 1\}$, with the convention that $\inf \emptyset = \infty$. Then $c_{6,1}'\in[0,\infty]$ and 
 \begin{equation}\label{01U}
 \liminf_{r\rightarrow0+}\sup_{|s|\le r}  \frac{ |U(x+s)-U(x)|} {r^{\alpha+1/2}/(\ln \ln 1/r)^{\alpha+1/2}}  = c_{6,1}', \ \ \  \ \ \ \text{a.s.}
 \end{equation}
 Moreover,  $c_{6,1}'$ does not depend on $x$ because of the stationarity of $U$. 
 
 Next, we use \eqref{01U} and Lamperti's transformation to show (\ref{Eq:Chung01law}). For any 
 $t > 0$ fixed, $0 < r < t$  and $|s|\le r$, we write 
\begin{equation}\label{01U2}
\begin{split}
 &Z(t+s) - Z(t) \\
  = &\,\big[(t+s)^H - t^H\big] U(\ln (t+s)) + t^H \big[U(\ln (t+s) ) - U(\ln t )\big]\\
  =&\,  \big[(t+s)^H - t^H\big] U(\ln (t+s)) + t^H \big[U\big(\ln t + \ln ( 1 +  s/t ) \big) - U(\ln t )\big].
 \end{split}
 \end{equation}
Since the first term is  Lipschitz in $s \in [-r, r]$ and $\ln \big( 1 + \frac s t) \big) \sim s/t$ as $s \to 0$, 
it can be verified that (\ref{Eq:Chung01law})  follows from (\ref{01U2}) and \eqref{01U} with $x = \ln t$.  
\end{proof}

It follows from Lemma \ref{lem:01lawZ} that Proposition \ref{lem:ZLIL} will be established if we show 
$c_{6,1}' \in (0, \infty)$. This is where Propositions \ref{lem SLND},  \ref{lem small2}, Lemma \ref{lem spectral}  
and the following version of Fernique's lemma from  \cite[Lemma 1.1, p.138]{JM1978} are needed. 
   
 \begin{lemma}\label{lem: Fern} Let $\{\xi(t)\}_{t\ge0}$ be a separable, centered, real-valued Gaussian 
 process. Assume that 
 $$
\mathbb E\left[\left(\xi(t+h)-\xi(t)\right)^2\right]\le \varphi(h)^2, \ \ \ t>0, \ \ h>0,
 $$
 for some continuous nondecreasing function $\varphi$ with $\varphi(0)=0$.  For any positive integer $k>1$ and 
 any positive constants $t, x$ and $\theta(p), p\in \mathbb N$,  we have
 \begin{equation*}
  \mathbb P\left\{\sup_{0\le s\le t}|\xi(s)-\xi(0)|>x\varphi(t)+\sum_{p=1}^{\infty} \theta(p)\varphi\left(tk^{-2^p} \right)  \right\}
 \le   k^2 e^{-x^2/2}+\sum_{p=1}^{\infty} k^{2^{p+1}}e^{-\theta(p)^2/2}.
 \end{equation*}
 \end{lemma}  
         
 We now give the proof of Proposition \ref{lem:ZLIL}.
\begin{proof}[Proof of Proposition \ref{lem:ZLIL}] Without loss of generality, 
we may assume  $t>1$.  We  first prove the lower bound.
For any integer $n \ge 1$, let $r_n = e^{-n}$. Let $0 < \delta <
c_{5,8}$ be a constant and consider the event
\[
A_n = \bigg\{ \sup_{|s|\le r_n} \big|Z(t+s)-Z(t)\big| \le
 \delta^{\alpha+1/2}\, t^{-\gamma/2}\, r_n^{\alpha+1/2}/(\ln \ln (1/r_n))^{\alpha+1/2}\bigg\}.
\]
Proposition \ref{lem small2} implies  that for  any $n\in \mathbb N$,
\begin{equation}\label{Eq:An1}
\mathbb P\{A_n\} \le \exp\bigg(- \frac{c_{{5,8}} } {\delta} \ln n \bigg)
= n^{- \frac{c_{{5,8}} }{\delta}}.
\end{equation}
Since $\sum_{n=1}^{\infty} \mathbb P\{A_n\} < \infty$, the Borel-Cantelli lemma
implies
\begin{equation}\label{Eq:LIL-lb1}
\liminf_{n\to \infty} \sup_{|s|\le r_n}\frac{ \big|Z(t+s)-Z(t)\big|} {r_n^{\alpha+1/2}/(\ln \ln (1/r_n))^{\alpha+1/2}}
\ge  \delta^{\alpha+1/2}\, t^{-\gamma/2}, \qquad \hbox{ a.s.}
\end{equation}
It follows from (\ref{Eq:LIL-lb1}) and a standard monotonicity
argument that
\begin{equation}\label{Eq:LIL-lb2}
\liminf_{r \to 0+}\sup_{|s| \le   r} \frac{ \big|Z(t+s)-Z(t)\big|}
{r^{\alpha+1/2}/(\ln \ln (1/r))^{\alpha+1/2}} \ge c_{{6,2}}\, t^{-\gamma/2}, \qquad \hbox{ a.s., }
\end{equation}
for some   constant $c_{6,2}\in (0,\infty)$ which is independent of $t>0$.

The upper bound is a little more difficult to prove due to the dependence structure of $Z$. 
In order to create independence, as in Tudor and Xiao \cite{TX2007}, we will make use 
of the following stochastic integral representation of $Z$:  
\begin{equation}\label{Eq:Rep1}
Z(t) = t^{H} \int_{\mathbb R} e^{i \lambda \ln t}\, 
W(d\lambda),\ \ \ \ \ t>0,
\end{equation}
where $H=\alpha-\gamma/2+1/2$. This follows from the spectral representation \eqref{eq U int} 
of $U $.

For every integer $n \ge 1$, we take
\begin{equation}\label{Eq:tn}
t_n = n^{-n} \ \ \ \hbox{ and }\ \ \ d_n = n^{n+1/2-\alpha}.
\end{equation}
  It is sufficient to prove that there exists a finite
constant $ c_{{6,3}}$ such that
\begin{equation}\label{Eq:UP}
\liminf_{n\to \infty}\sup_{|s|\le 
t_n} \frac{ \big|Z(t+s)-Z(t)\big|} {t_n^{{\alpha+1/2}}/(\ln \ln (1/t_n))^{\alpha+1/2}}
\le c_{{6,3}}t^{-\gamma/2} \qquad \hbox{ a.s.}
\end{equation}

Let us define two Gaussian processes $Z_n$ and $\widetilde{Z}_n$ by
\begin{equation}\label{Eq:Xn1}
Z_n(t) = t^{H} \int_{|\lambda| \in (d_{n-1}, d_n]} e^{i \lambda
\ln t}\, W(d\lambda)
\end{equation}
and
\begin{equation}\label{Eq:Xn2}
\widetilde{Z}_n(t) = t^{H} \int_{|\lambda| \notin (d_{n-1}, d_n]}
e^{i \lambda \ln t}\, W(d\lambda),
\end{equation}
respectively. Clearly $Z(t) = Z_n(t) + \widetilde {Z}_n(t)$ for all $t >0$. It is 
important to note that the Gaussian processes $Z_n\, (n = 1, 2, \ldots)$ are 
independent and, moreover, for every $n \ge 1,$ $Z_n$ and $\widetilde{Z}_n$ 
are independent as well.

Denote $h(r) = r^{\alpha+1/2}\, \big(\ln \ln (1/r)\big)^{-(\alpha+1/2)}$. We make
the following two claims:
\begin{itemize}
\item[(i).]\ There is a constant $\delta > 0$ such that
\begin{equation}\label{Eq:UP1}
\sum_{n=1}^\infty \mathbb P\bigg\{\sup_{|s|\le t_n} \big|Z_n(t+s)-Z_n(t)\big|
\le \delta^{\alpha+1/2}\,t^{-\gamma/2}\, h(t_n)\bigg\} = \infty.
\end{equation}
\item[(ii).]\ For every $\varepsilon> 0$,
\begin{equation}\label{Eq:UP2}
\sum_{n=1}^\infty \mathbb P\bigg\{\sup_{|s| \le   t_n}
\big|\widetilde{Z}_n(t+s)-\widetilde{Z}_n(t)\big| > \varepsilon\,t^{-\gamma/2}\, h(t_n)\bigg\} < \infty.
\end{equation}
\end{itemize}
Since the events in (\ref{Eq:UP1}) are independent, we see that (\ref{Eq:UP}) follows 
from (\ref{Eq:UP1}), (\ref{Eq:UP2}) and a standard Borel-Cantelli argument.

It remains to verify the claims (i) and (ii) above. By  Proposition \ref{lem small2}
and Anderson's inequality  \cite{And1955}, we have
\begin{equation}\label{Eq:UP3}
\begin{split}
&\mathbb P\bigg\{\sup_{|s| \le t_n} \big|Z_n(t+s)-Z_n(s)\big| \le \delta^{\alpha+1/2}\, t^{-\gamma/2}\,
h(t_n)\bigg\}\\
 &\ge \,\mathbb P\bigg\{\sup_{|s| \le t_n} \big|Z(t+s)-Z(s)\big| \le \delta^{\alpha+1/2}\, t^{-\gamma/2}\, h(t_n)\bigg\}\\
 &\ge \, \exp\Big(- \frac{c_{{5,7}}} {\delta} \ln (n \ln n) \Big)\\
&=\, \big(n \ln n\big)^{-\frac{ c_{5,7}}{\delta}}.
\end{split}
\end{equation}
Hence (i) holds for $\delta \ge  c_{5,7}$.

In order to  prove (ii),  we divide it into two terms:  For any $|s|<t_n$,
\begin{equation}\label{Eq:J1}
\begin{split}
\mathbb E\Big(\big(\widetilde{Z}_n(t+s) - \widetilde{Z}_n(t)\big)^2\Big) &=
\int_{|\lambda|\le d_{n-1}} \big| (t+s)^{H} \, e^{i \lambda \ln (t+s)} -
t^{H} \, e^{i \lambda \ln t}\big|^2 \,f_U(\lambda)\, d\lambda\\
&\quad + \int_{|\lambda| > d_{n}} \big| (t+s)^{H} \, e^{i \lambda
\ln (t+s)} - t^{H} \, e^{i \lambda \ln t}\big|^2 \,f_U(\lambda)\,
d\lambda\\
&=: \mathcal J_1 + \mathcal J_2.
\end{split}
\end{equation}
The second term is easy to estimate:  For any $|s|\le t_n$,  there exists $c_{6,4}\in (0,\infty)$ satisfying that 
\begin{equation}\label{Eq:J2}
\mathcal {J}_2 \le 2\,\left(t^{2H}+ (t+s)^{2H}\right)\, \int _{|\lambda| >
d_{n}}\,f_U(\lambda)\, d\lambda\le c_{{6,4}}\,   \,
n^{-(2\alpha+1)(n+1/2-\alpha)},
\end{equation}
where the last inequality follows from \eqref{Eq: sp2}.

For the first term $\mathcal J_1$, we use the following  elementary inequalities:
\begin{align*}
\left|(x+y)^H-x^H\right|\le |H| \max\left\{x^{H-1}, (x+y)^{H-1}\right\}|y|, \ \ \ \text{ for all }  0\le |y|<x,
\end{align*}
$1 - \cos x
\le x^2$ for all $x\in \mathbb R$
 and $\ln (1+x)\le x$ for all $x\ge0$ to derive that for any $|s|\le t_n$,
\begin{equation*} 
\begin{split}
\mathcal J_1 =&\, \int_{|\lambda|\le d_{n-1}} \left[\left((t+s)^{H} -
t^{H}\right)^2  + 2 (t+s)^{H}\,t^{H} \Big( 1 - \cos\Big( \lambda \ln 
\frac{t+s}t\Big)\Big)\right] \,f_U(\lambda)\, d\lambda\\
  \le &\, \frac{1}{H^2}\max\left\{ t^{2(H-1)},  (t+s)^{2(H-1)}  \right\}s^{2}  \int_{\mathbb R}
\,f_U(\lambda)\, d\lambda \\
& \quad \ + 2(t+s)^{H}\, t^H\, \ln ^2\left(1+\frac{s}t\right)\,
\int_{|\lambda|\le d_{n-1}}  \lambda^2
\,f_U(\lambda)\, d\lambda\\ 
  \le &\, c_{6,5}  \left( \int_{\mathbb R}
\,f_U(\lambda)\, d\lambda+ \int_{|\lambda|\le d_{n-1}}  \lambda^2
\,f_U(\lambda)\, d\lambda \right)n^{-2n}.
\end{split}
\end{equation*}
Here $c_{6,5}$ is in $(0,\infty)$. 
By  \eqref{Eq: sp1} and  \eqref{Eq: sp2}, we know 
$$
\int_{\mathbb R} \,f_U(\lambda)\, d\lambda+ \sup_{n\ge1}\int_{|\lambda|\le d_{n-1}}  \lambda^2
\,f_U(\lambda)\, d\lambda <\infty. 
$$
Notice that $2n\ge (2\alpha+1)(n+1/2-\alpha)$ for any $\alpha\in (-1/2+\gamma/2, 1/2)$. Thus, 
\begin{align}\label{Eq:J1b}
\mathcal J_1 \le c_{6,6} n^{-(2\alpha+1)(n+1/2-\alpha)},
\end{align}
for some constant $c_{6,6}\in (0,\infty)$. 
 
Put $\delta=(2\alpha+1)(1/2-\alpha)$. By Lemma \ref{lem expect},  \eqref{Eq:J1}, \eqref{Eq:J2}  
and \eqref{Eq:J1b},  there exists  a constant $K>0$ such that
for $0\le h\le t_n$,
\begin{align*}
\varphi_{n}(h)^2:=K\min\left\{h^{2\alpha+1},\,  n^{-(2\alpha+1)n-\delta} \right\}\ge 
\mathbb E\left[\big(\widetilde Z_n(t+h) -\widetilde Z_n(t) \big)^2  \right].
\end{align*} 
 
Put $x_n=(8\ln n)^{1/2}$. Given $\varepsilon>0$,   define 
\begin{align*}
\theta_n(p)=\varepsilon (p+1)^{-2} h(t_n)/\varphi_n\left(t_n n^{-2^p}\right)\ \ \ \ \text{for  all } p\ge1.
\end{align*} 
For large enough $n$, we have
$$
\theta_n(p)>4(\ln n)^{1/2} 2^{p/2}\ \ \ \text{for all } p\ge1
$$
and 
$$
x_n \varphi_n(t_n)+\sum_{p=1}^{\infty} \theta_n(p) \varphi_n\left(t_n n^{-2^p} \right)<\varepsilon h(t_n).
$$
Since 
$$
\sum_{n=1}^{\infty}n^2e^{-x_n^2/2}+\sum_{n=1}^{\infty}\sum_{p=1}^{\infty} n^{2^{p+1}}e^{-8(\ln n)2^p}<\infty,
$$
by applying  Lemma \ref{lem: Fern} with $\xi(s)=\widetilde Z_n(t\pm s)-\widetilde Z_n(t)$ for $0\le s\le t_n$, we 
obtain that 
$$
\sum_{n=1}^{\infty}\mathbb P\bigg\{\sup_{|s|\le t_n}\big|\widetilde Z_n(t+s)-\widetilde Z_n(t)\big|
>\varepsilon\, t^{-\gamma/2}\, h(t_n)  \bigg\}<\infty.
$$  
This proves (\ref{Eq:UP2}) and hence Proposition \ref{lem:ZLIL}.
\end{proof}

\begin{remark} {\rm In light of Proposition \ref{lem:ZLIL}, it is natural to study Chung's LIL of $Z$ at the origin.
While doing so, we found that there is an error in the proof of Theorem 3.1 in Tudor and Xiao \cite{TX2007}, which 
gives Chung's  LIL for bifractional Brownian motion at the origin. More precisely, the inequality (3.30) in \cite{TX2007} 
does not hold. Hence, Theorem 3.1 in \cite{TX2007} should be modified as a Chung's LIL at $t>0$, which is similar
to Theorem 1.4 and Proposition 6.1 in the present paper. 

It turns out that Chung's LIL at the origin for self-similar Gaussian processes that do not have stationary 
increments such as  GFBM $X$, the generalized Riemann-Liouville FBM $Z$, and bifractional Brownian motion 
is quite subtle. A different method than that in the proof of  Proposition \ref{lem:ZLIL} or that in \cite{TX2007} is 
needed for proving the desired upper bound. Recently, this problem has been studied in the subsequent paper 
\cite{WX2021} by modifying the arguments of Talagrand \cite{Tal96}. } 
 \end{remark} 
   
Similarly, for $\alpha\in (1/2,\, 3/2)$, we have the following Chung's LIL for $Z'$.  

\begin{proposition}\label{lem:ZLIL2}
Assume $\alpha\in(1/2,\, 3/2)$. There exists a constant $c_{6,7} \in(0,\infty)$  such that for any $t>0$, 
\begin{align}\label{eq ZLIL2}
\liminf_{r\rightarrow0+}\sup_{0\le s\le r} \frac{ |Z'(t+s)-Z'(t)|} {r^{\alpha-1/2}/(\ln \ln 1/r)^{\alpha-1/2}}
=c_{6,7}t^{-\gamma/2} \  \ \ \ \text{a.s.}
\end{align}  
\end{proposition}
\begin{proof}  Recall from \eqref{eq Z 1} that, for $\alpha\in(1/2,\, 3/2)$, the derivative process 
$\{Z'(t)\}_{ t \ge 0}$ is a generalized Riemann-Liouville FBM with indices $\alpha'=\alpha-1\in (-1/2,\ 1/2)$ 
and $\gamma\in (0,1)$. Hence the proof of \eqref{eq ZLIL2} follows the same line as in the proof of 
Proposition \ref{lem:ZLIL}. We omit the details. 
\end{proof}

 \section{Proofs of the main theorems} 

In this section we prove our main results for GFBM $X$ stated in Section 1.   
     
\subsection{Proof of Theorem  \ref{thm uniform}}
(a). Assume that $\alpha\in (-1/2+\gamma/2,1/2)$.  By Proposition \ref{prop Y}, 
we know that for any $0<a<b<\infty$,  $Y(t)$ is continuously differentiable on $[a, b]$. Then
 $$
     \lim_{\ep\rightarrow0} \sup_{a\le t\le b}\sup_{0\le h\le \ep}\frac{|Y(t+h)-Y(t)|}{h^{\alpha+\frac12}\sqrt{\ln h^{-1}}}=0.
 $$
By \eqref{eq decom}, to prove \eqref{eq unif X}, it is sufficient to prove that for any $0<a<b<\infty$,  
     \begin{align}\label{eq unif Z}
     \lim_{\ep\rightarrow0} \sup_{a\le t\le b}\sup_{0\le h\le \ep}\frac{|Z(t+h)-Z(t)|}{h^{\alpha+\frac12}\sqrt{\ln h^{-1}}}
     =c_{7,1},
 \end{align}
 where $c_{7,1}$ is a positive constant  satisfying 
  \begin{align}\label{eq kappa}
 c_{7, 2}:= \sqrt{ \frac{2 }{ (\alpha+1/2)^2 b^{\gamma}c_{3,6}}  }\le c_{7,1}\le 
  \sqrt{ \frac{2c_{3,6}}{ (\alpha+1/2)c_{3,5}} }=:c_{7,3}.
  \end{align}
 Here, $c_{3,5},\, c_{3,6}$   are constants given in \eqref{eq expect1}.
  
 For any $\ep>0$, let 
  $$
  J(\ep):=\sup_{a\le t\le b}\sup_{0\le h\le \ep}\frac{|Z(t+h)-Z(t)|}{h^{\alpha+\frac12}\sqrt{\ln h^{-1}}}.
  $$
 Since $\ep\mapsto J(\ep)$ is non-decreasing, the limit in the left-hand side of \eqref{eq unif Z} exists almost 
 surely. Moreover,  the zero-one law for the uniform modulus of continuity in \cite[Lemma 7.1.1]{MR} implies 
 that the limit in \eqref{eq unif Z} is a constant almost surely. Hence, it remains to prove that with probability one, 
  \begin{align}\label{eq upper}
  \lim_{\ep\rightarrow0+}J(\ep)\le c_{7,3} \ 
  \end{align}
  and  
 \begin{align}\label{eq lower}
  \lim_{\ep\rightarrow0+}J(\ep)\ge c_{7,2}.
  \end{align}
 
The proof of \eqref{eq upper} is standard. It follows from \eqref{eq expect1} and the metric-entropy bound (cf. 
Theorem 1.3.5 in \cite{AT07}), or one may prove this directly by applying Fernique's inequality in Lemma \ref{lem: Fern} 
and the Borel-Cantelli lemma. The lower bound in \eqref{eq lower} follows from
the one-sided SLND \eqref{eq one-sided} and Theorem 4.1 in \cite{MWX}. This proves (a).
  
(b).  When $\alpha=1/2$, we see that \eqref{eq expect12} holds for all $s, t \in [a, b]$. This implies that the canonical 
metric of $Z$ satisfies 
$$d_Z(s, t) \asymp |s-t| \sqrt{1 + \big| \ln  |s-t| \big|}  $$ 
on $[a, b]$. Hence,  for any $\varepsilon > 0$ small, 
the covering number $N_\varepsilon([a, b])$ of $[a, b]$ under $d_Z$ satisfies
\begin{equation}\label{Covering2}
N_\varepsilon([a, b]) \le c_{7,4} \varepsilon^{-1}  \big(\ln  \varepsilon^{-1} \big)^{1/2},
\end{equation}
where $c_{7,4}>0$ is a finite constant.  It follows from Theorem 1.3.5 in \cite{AT07} that almost surely for all $\delta > 0$
small enough, 
\begin{equation}\label{Dudley}
\begin{split}
\sup_{a\le t\le b}\sup_{d_Z(t, t+h) \le \delta} \big|Z(t+h)-Z(t) \big| &\le c \int_0^\delta \sqrt{\ln N_\varepsilon([a, b])}\, d\ep
\le c \delta \sqrt{\ln  {\delta}^{-1} }.
\end{split}
\end{equation}
Notice that, by \eqref{eq expect12}, $d_Z(t, t+h) \le \delta$ is compatible to $|h| \sqrt{\ln  |h|^{-1}} \le \delta$, 
up to a constant factor. Hence,  \eqref{Dudley} implies that 
\begin{equation}\label{Dudley2}
\begin{split}
\sup_{a\le t\le b}\sup_{0 \le h \le r} \big|Z(t+h)-Z(t) \big| \le  c\, r \ln  r^{-1}.
\end{split}
\end{equation}
Now it can be seen that  \eqref{eq unif X2} follows from \eqref{Dudley2}.

 (c).  The proof of \eqref{eq unif X'} is similar: the uniform modulus of continuity of $X'$ on $[a, b]$ is 
 the same as that of $Z'$ on $[a, b]$, and the latter can be derived from \eqref{eq expect3} in
 Lemma \ref{lem expect} and the one-sided SLND (\ref{eq one-sided2}) in Proposition \ref{lem SLND}. 
 We omit the details here. The proof is complete. \qed

 \subsection{Proofs of Theorems \ref{thm Xsb} and \ref{Prop:sbX}}     
  
Similarly to the proof of Proposition \ref{lem small}, by using Lemmas \ref{lem Y moment} and \ref{Lem:Ta93}, 
we obtain the following lower bounds for the small ball probabilities of $Y$.
 
 \begin{lemma}\label{Lem: Ysb}
 \begin{itemize}
 \item[(a).]
 Assume  $\alpha\in (-1/2+\gamma/2, 1/2+\gamma/2)$. Then there exists a  
 constant $c_{7,5}\in (0, \infty)$ such that for all $r>0, \varepsilon>0$, 
  \begin{align}\label{Eq:Ysb}
 \mathbb P\bigg\{\sup_{0\le s\le r} |Y(s)|\le  \varepsilon \bigg\} \ge \exp\bigg(- c_{7,5}   \, {r^{H}}\, \varepsilon^{-1} \bigg).
\end{align} 
  \item[(b).]
 Assume   $\alpha\in (-1/2+\gamma/2, 1/2+\gamma/2)$.  Then there  
 exists a  constant $c_{7,6}\in (0,\infty)$   such that for all $t>0, 0<\varepsilon<r<t/2$, 
  \begin{align}\label{Eq:Ysc}
 \mathbb P\bigg\{\sup_{|s|\le r} |Y(t+s)-Y(t)|\le  \varepsilon \bigg\} \ge \exp\bigg(- c_{7,6} \, r \, t^{H-1}  \, \varepsilon^{-1}\bigg).
 \end{align}
 \end{itemize}
 \end{lemma}
 
 \begin{proof}[Proof of Theorem \ref{thm Xsb}]  By the self-similarity of $X$, we know that \eqref{eq small X} 
 is equivalent to the following: For any $0<\varepsilon<1$,
  \begin{align}\label{eq small X7}
\exp\bigg(- \kappa_{4}  \Big(\frac{1}{\varepsilon}\Big)^{\frac{1}{\alpha+1/2}} \bigg)
\le  \mathbb P\bigg\{\sup_{s\in [0,1]} |X(s)|\le  \varepsilon \bigg\} 
\le  \exp\bigg(- \kappa_5 \Big(\frac{1}{\varepsilon}\Big)^{\frac{1}{\alpha+1/2}} \bigg).
 \end{align} 
  
By  \eqref{Eq:sb1},  \eqref{Eq:Ysb}   and the independence of $Y$ and $Z$, we have 
\[ \begin{split}
 \mathbb P\bigg\{\sup_{s\in [0,1]} |X(s)|\le \varepsilon  \bigg\}\ge &\,
 \mathbb P\bigg\{\sup_{s\in [0,1]} |Y(s)|\le \varepsilon/2  \bigg\}\cdot 
 \mathbb P\bigg\{\sup_{s\in [0,1]} |Z(s)|\le \varepsilon/2  \bigg\}\notag\\
 \ge &\, \exp\left(-  c_{7,5} \frac{2}{\varepsilon}\right) \cdot \exp\bigg(-c_{5,1}\Big( \frac{2}{\varepsilon} \Big)^{\frac{1}{\alpha+1/2}}\bigg) \\
   \ge &\, \exp\bigg(-    \left(2 c_{7,5}+2^{\frac{1}{\alpha+1/2}} c_{5,1}\right)\Big( \frac{1}{\varepsilon} \Big)^{\frac{1}{\alpha+1/2}}\bigg),
 \end{split}
 \]
 where we have used the facts that $\frac{1}{\varepsilon}\le \left(\frac{1}{\varepsilon} \right)^{\frac{1}{\alpha+1/2}}$ 
 as $\varepsilon \in (0,1)$ and $\alpha<1/2$. This proves the lower bound in \eqref{eq small X7}.
 
 On the other hand, by using the  Anderson inequality \cite{And1955} and   \eqref{Eq:sb1},  we have 
\[  \begin{split}
 \mathbb P\bigg\{\sup_{s\in [0,1]} |X(s)|\le \varepsilon  \bigg\} 
 = &\,\mathbb E\left[\mathbb P\bigg\{\sup_{s\in [0,1]} |Y(s)+Z(s)|\le \varepsilon \big|Y(s), s\in [0,1] \bigg\}  \right]\\
 \le &\, \mathbb P\bigg\{\sup_{s\in [0,1]} |Z(s)|\le \varepsilon  \bigg\}
 \le \, \exp\bigg(- c_{5,2}  \Big(\frac{1}{\varepsilon}\Big)^{\frac{1}{\alpha+1/2}} \bigg).
 \end{split}
 \]
This proves the upper bound in \eqref{eq small X7}. The proof is complete.
 \end{proof}

 Similarly to the proof of Theorem \ref{thm Xsb},  by using Proposition \ref{lem small2} and Lemma 
 \ref{Lem: Ysb}, we can prove Theorem \ref{Prop:sbX}. The detail is omitted.      
      
 \subsection{Proof of Theorem \ref{thm LIL0}}  
By Proposition \ref{prop Y}, $Y(t)$ is continuously differentiable on $[a, b]$ for any $0< a< b< \infty$. 
Hence Chung's LILs of $X$ (or $X'$ when it exists) at a fixed point $t > 0$ is the same as that of 
$Z$ (or $Z'$) at $t$. Therefore (\ref{eq LIL01}) and (\ref{eq LIL01b}) follows from \eqref{eq ZLIL1} 
and \eqref{eq ZLIL2},  respectively.
\qed

\begin{remark}\label{Re:refine}{\rm The small ball probabilities in Theorems \ref{thm Xsb} and \ref{Prop:sbX}
and the one-sided SLND in Proposition \ref{lem SLND} can be applied as in \cite{Xiao1997} to refine the results in 
\eqref{Eq:Gr} and \eqref{Eq:Level} by determining the exact Hausdorff measure functions for the graph and 
level set of GFBM $X$.  We refer to Talagrand \cite{Tal95,Tal98} for further ideas on investigating other related 
fractal properties of $X$. 
}
\end{remark}
      
 \subsection{Proof of Theorem \ref{thm LIL2} and the time inversion of  $X$}   
  
 \begin{proof}[Proof of Theorem \ref{thm LIL2}] As in the proof of Theorem \ref{thm LIL0}, it is sufficient to prove 
\eqref{eq LIL21}-\eqref{eq LIL21b} for $Z$ and $Z'$, respectively. 

First we consider the case when  $\alpha \in (-1/2+\gamma/2,1/2)$. It follows from \eqref{01U2} that for any fixed $t > 0$,  
 \begin{equation}\label{Eq:LIL-U}
 \limsup_{r\rightarrow0+} \sup_{ |h|\le r}\frac{ |Z(t+h)-Z(t)|} {r^{\alpha+1/2} \sqrt{\ln \ln 1/r}} 
 =  t^{-\gamma/2} \limsup_{r\rightarrow0+} \sup_{ |h|\le r} \frac{ |U(x + h )-U(x)|} {r^{\alpha+1/2} \sqrt{\ln \ln 1/r}}, 
\end{equation}
where $x = \ln t$. By Lemma \ref{lem r}  we see that the Gaussian process $\{U(s) - U(0)\}_{s \in \mathbb R} $ 
satisfies the conditions of Theorem 5.1 in \cite{MWX}  with $N=1$ and $\sigma(s, t) \asymp |s-t|^{\alpha+1/2}$ on 
any compact interval of $(0, \infty)$. Consequently,  there exists a constant $\kappa_{12} \in (0, \infty)$ which 
does not depend on $x$ such that
\begin{equation}\label{Eq:LIL-U2}
 \limsup_{r\rightarrow0+} \sup_{ |h|\le r} \frac{ |U(x + h )-U(x)|} {r^{\alpha+1/2} \sqrt{\ln \ln 1/r}} = \kappa_{12}, \ \ \ \hbox{a.s.}
\end{equation}
Combining (\ref{Eq:LIL-U}) and \eqref{Eq:LIL-U2} yields \eqref{eq LIL21} for $Z$. 

When  $\alpha = 1/2$, we see from Lemma \ref{lem r} that Condition (A1) in \cite{MWX} is satisfied with $N=1$ and 
$\sigma(s, t) \asymp |s -t| \sqrt{1 + \big|\ln |s-t| \big|}$ on any compact interval of $(0, \infty)$. This case is not explicitly 
covered by Theorem 5.1 in \cite{MWX}. However, an examination of its proof shows that we have
\begin{equation}\label{Eq:LIL-U3}
 \limsup_{r\rightarrow0+} \sup_{ |h|\le r} \frac{ |U(x + h )-U(x)|} {r \sqrt{\ln (1/r) \ln \ln 1/r}} = \kappa_{13}, \ \ \ \hbox{a.s.}
\end{equation}
for a constant $\kappa_{13} \in (0, \infty)$ that does not depend on $x$.  
Therefore, by combining \eqref{Eq:LIL-U} and \eqref{Eq:LIL-U3}     we obtain \eqref{eq LIL21a} 
for $Z$.

Finally, when $\alpha\in(1/2, \, 1/2+\gamma/2)$, the proof of \eqref{eq LIL21b} for $Z'$ is similar to that of 
\eqref{eq LIL21} for $Z$ given above. There is no need to repeat it. The proof of Theorem \ref{thm LIL2} is complete.
 \end{proof}
 
  The following  result is the {\it time inversion} property of GFBM $X$. It allows us to improve 
 slightly Theorem 6.1 in \cite{IPT2020},  where the law of the iterated logarithm of $X$ 
at the origin was proved under the extra condition of $\alpha>0$.

 \begin{proposition}\label{prop LIL0} 
 Let $X=\{X(t)\}_{t\ge0}$ be a GFBM with parameters $\gamma \in (0, 1)$ and  $\alpha\in(-1/2+\gamma/2, 1/2+\gamma)$.
 Define the Gaussian process  $\widetilde X=\{\widetilde X(t)\}_{t\ge0}$ by
 $$\widetilde X(0)=0, \ \ \  \widetilde X(t)=t^{2H}X\left(1/t\right),\ \ \ \ \  t>0.$$  
 Then $X$ and  $\widetilde X$ have the same distribution. Consequently,  there exists a constant $\kappa_{15}
  \in(0,\infty)$ such that 
   \begin{align}\label{eq LIL0}
   \limsup_{r \to 0+} \frac{|X(r)|} {r^H\sqrt{\ln \ln r^{-1}}}=\kappa_{15},\ \  \ \ \ \hbox{a.s.}
   \end{align}  
 \end{proposition}
 \begin{proof} 
  Since $X$ and $\widetilde X$ are centered Gaussian processes,  it is sufficient to check their covariance functions. 
 For any $0< s\le t$,
 \begin{align*}
   \mathbb E\big[\widetilde X(s)\widetilde X(t)\big] 
      =&\,  (st)^{2H}\int_{\mathbb R} \left( \Big(\frac1s-u\Big)_+^{\alpha}-(-u)_+^{\alpha}  \right)
      \left( \Big(\frac1t-u\Big)_+^{\alpha}-(-u)_+^{\alpha}  \right) |u|^{-\gamma}du\notag\\
       =& \, \int_{\mathbb R} \left( \left(t-v\right)_+^{\alpha}-(-v)_+^{\alpha}  \right)
       \left( \left(s-v\right)_+^{\alpha}-(-v)_+^{\alpha}  \right) |v|^{-\gamma}dv\notag\\
=& \, \mathbb E\left[ X(s)X(t)\right],
\end{align*}
where  a change of variable $u=v/(st)$ is used in the second step. This proves the first part of Proposition \ref{prop LIL0}.
 
The time inversion property, together with the LIL of $X$ at infinity in \cite[Theorem 5.1]{IPT2020}, implies 
\eqref{eq LIL0}. The proof is complete. 
  \end{proof}
 
\subsection{Proof of Theorem \ref  {thm tangent} }  
By Proposition \ref{prop Y}, $Y(t)$ is continuously differentiable in $(0,\infty)$.
Hence  for every  $\alpha\in(-1/2+\gamma/2, 1/2)$, $t>0$ and for every  compact set $\mathcal K\subset \mathbb R_+$,  
  \begin{equation}\label{eq U}
 \lim_{u\rightarrow 0+} \sup_{\tau \in \mathcal K} \left|\frac{ X(t+u\tau)-X(t)}{u^{\alpha+1/2}}-
  \frac{ Z(t+u\tau)-Z(t)}{u^{\alpha+1/2}}\right|=0,  
 \ \ \     \, \,   \text{ a.s.}  
 \end{equation} 
 Therefore,  $X$ has the same tangent process as that of 
$Z$  at $t\in (0,\infty)$, and Theorem \ref{thm tangent} follows from Proposition \ref{Prop tangent Z}. \qed

 \vskip0.5cm

\noindent{\bf Acknowledgments}: 
The authors are grateful to the anonymous referees for their constructive comments and corrections which have lead to 
significant improvement of this paper.  
Part of the research in this paper was conducted during R. Wang's visit to Michigan State University (MSU).  
He thanks the Department of Statistics and Probability at MSU for the hospitality and the financial support from 
the CSC Fund, NNSFC 11871382, and the Fundamental Research Funds for the Central Universities 2042020kf0031. 
The research of Y. Xiao is partially supported by the NSF grant DMS-1855185.

\medskip

\bigskip

\end{document}